\documentclass{ws-join}

\usepackage{amsmath,amssymb}
\usepackage{graphicx}
\usepackage{latexsym}
\usepackage{wasysym}
\usepackage{color}
\newtheorem{observation} {Remark}[section]
\newcommand{\PartIntSup}[1]{\left\lceil #1\right\rceil}
\newcommand{\PartIntInf}[1]{\left\lfloor #1\right\rfloor}
\begin{document}

\markboth{Huc, Sau, \v{Z}erovnik}{$(\ell,k)$-Routing on
Plane Grids}

%\catchline{}{}{}{}{}
%%%%%%%%%%%%%%%%%%%%%%%%%%%%%%%%%%%%%%%%%%%%%%%%%%%%%%%%%%%%%%%%%%%%
\title{$(\ell,k)$-Routing on Plane Grids}
\date{    } 

\author{Florian Huc}
\address{Centre Universitaire d'Informatique, Battelle b\^atiment A,\\ route de Drize 7, 1227 Carouge, Geneva, Switzerland\\
\email{florian.huc@unige.ch}}

\author{Ignasi Sau}
\address{Mascotte Project, CNRS/INRIA/UNSA/I3S\\
2004 route des Lucioles\\
B.P. 93 F-06902 Sophia-Antipolis Cedex, France\\
\&\\
Graph Theory and Combinatorics group,\\
MA4, UPC, Barcelona, Spain\\
\email{Ignasi.Sau@sophia.inria.fr}}

%V-O And ...> \&

\author{Janez \v{Z}erovnik}
\address{Institute of Mathematics, Physics and Mechanics (IMFM),\\ Jadranska 19, Ljubljana\\
\&\\
FME, University of Ljubljana,\\ A\v{s}ker\v{c}eva 6
, Ljubljana, Slovenia\\
\email{janez.zerovnik@imfm.uni-lj.si}
}

\maketitle

\begin{abstract}
The packet routing problem plays an essential role in communication networks. It involves how to transfer data from some origins to some destinations within a reasonable amount of time. In the $(\ell,k)$-routing problem, each node can send at most $\ell$ packets and receive at most $k$ packets. Permutation routing is the particular case $\ell=k=1$. In the $r$-central routing problem, all nodes at distance at most $r$ from a fixed node $v$ want to send a packet to $v$.

In this article we study the permutation routing, the $r$-central routing and the general $(\ell,k)$-routing problems on plane grids, that is square grids, triangular grids and hexagonal grids. We use the \emph{store-and-forward} $\Delta$-port model, and we consider both full and half-duplex networks. We first survey the existing results in the literature about packet routing, with special emphasis on $(\ell,k)$-routing on plane grids. Our main contributions are the following:

 1. Tight permutation routing algorithms on full-duplex hexagonal grids, and
half duplex triangular and hexagonal grids.
%\item[2.]

2. Tight $r$-central routing algorithms on  triangular and hexagonal grids.

3. Tight $(k,k)$-routing algorithms on square, triangular and hexagonal grids.

%\item[4.]
 4. Good approximation algorithms (in terms of running time) for $(\ell,k)$-routing on square, triangular and hexagonal grids, together with new lower bounds on the running time of any algorithm using shortest path routing.
%\end{itemize}

 These algorithms are all completely distributed, i.e., can be implemented independently at each node. Finally, we also formulate the $(\ell,k)$-routing problem as a \textsc{Weighted Edge Coloring} problem on bipartite graphs.\\
\end{abstract}


\section{Introduction}
\label{sec:intro}

In telecommunication networks, it is essential to be able to route
communications as quickly as possible. In this context, the
\emph{packet routing} problem plays a capital role. In this
problem we are given a network and a set of packets to be routed
through the nodes and the edges of the network graph. A packet is
characterized by its origin and its destination. We suppose that
an edge can be used by no more than one packet at the same time.
The objective is to find an algorithm to compute a schedule to
route all packets while minimizing the total delivery time. This
problem has been widely studied in the literature under many
different assumptions. In 1988, in their seminal article
\citelow{LMR88,Leighton_C_D_94},  Leighton, Maggs and Rao proved
the existence of a schedule for routing any set of packets with
edge-simple paths on a general network, in optimal time of
$\mathcal{O}(C+D)$ steps. Here $C$ is the congestion (maximum
number of paths sharing an edge) and $D$ the dilation (length of
the longest path) and it is assumed that the paths are given a
priori. The proof of \cite{Leighton_C_D_94} used Lov\'asz Local
Lemma and was non constructive. This result was further improved
in \cite{Leighton_fast_95} where the same authors gave an explicit
algorithm, using the Beck's constructive version of the Local
Lemma.

These algorithms to compute the optimal schedule are centralized.
Then in \cite{Universal97} Ostrovsky and Rabani gave a distributed
randomized algorithm running in $\mathcal{O}(C+D +
\log^{1+\epsilon}(n))$ steps (see Section \ref{sec:general} for
more references).

Although these results are asymptotically tight, they deal with a general
network, and in many cases it is possible to design more efficient
algorithms by looking at specific packet configurations or network
topologies. For instance, is it natural to bound the maximum number of
messages that a node can send or receive. We focus on this point in
Section \ref{sec:problem}, where we will formally define the problem studied in this paper.

\noindent On the other hand, the network considered  plays a major role on the
quality and the simplicity of the solution. For example, in a radio
wireless environment, cellular networks are usually modeled by a
\emph{hexagonal grid} where nodes represent base stations. The cells of
the hexagonal grids have good diameter to area ratio and still have a
simple structure. If centers of neighboring cells are connected, the
resulting graph is called a \emph{triangular grid}. Notice that hexagonal
grids are subgraphs of the triangular grid. We will talk about such
networks in Section \ref{sec:previous}. In this paper we focus on the
%V- modified
study of the $(\ell,k)$-routing problem in convex subgraphs, i.e., subgraphs of the square,
triangular and hexagonal grids which contain all shortest paths between all pairs of nodes.

%\newpage

\subsection{General Results on Packet Routing}
\label{sec:general} In this section we provide a fast overview of the
state-of-the-art of the general packet routing problem, in both the
off-line and on-line settings in Sections \ref{sec:off} and \ref{sec:on}
respectively, focusing mostly on the later. We begin by recalling three
classical lower bounds for the packet routing problem.

\subsubsection{Classical lower bounds}
\label{sec:bounds}  In the packet routing problem, there are three
classical types of lower bounds for the running time of any algorithm:

\begin{itemize}

\item[1.] \textbf{Distance bound}: the longest distance over the paths of all
packets (usually called \emph{dilation} and denoted by $D$) constitutes a lower
bound on the number of steps required to route all the packets.

\item[2.] \textbf{Congestion bound}: the \emph{congestion} of an edge of the network is
defined as the number of paths using this edge. The greatest
congestion over all the edges of the network (called \emph{congestion} and denoted by $C$) is also a lower bound on the number of steps, since at each step an edge can be used by at
most one packet.

\item[3.] \textbf{Bisection bound}: Let $G=(V,E)$ be the graph which
models the network, and $F \subseteq E$ be a cut-set disconnecting $G$ into
two components $G_1$ and $G_2$. Let $m$ be the number of packets with
origin in $G_1$ and destination in $G_2$. The number of routing
steps used by any algorithm is at least $\PartIntSup{\frac{m}{|F|}}$.

\end{itemize}
\subsubsection{Off-line routing}
\label{sec:off}

Given a set of packets to be sent through a network, a \emph{path system}
is defined as the union of the paths that each packet must follow. For a
general network and any set of $n$ demands, we have seen in Section
\ref{sec:bounds} that the dilation and the congestion provide two lower
bounds for the routing time. This proves that the $dilation + congestion$
of a paths system used for the routing procedure is a lower bound of twice
the routing time. In a celebrated paper, Leighton, Maggs and Rao proved
the following theorem:
%\begin{theorem}[\cite{LMR88}]
%For any set of requests, there is an {\bf off-line} routing protocol that
%needs $\mathcal O(C+D)$ steps to route all the requests, where $C+D$ is
%the minimum $congestion+dilation$ over all the possible path systems.
%\end{theorem}

\begin{theorem}[\cite{LMR88}]
\label{teo:C+D} For any set of requests and a path system for these
requests, there is an {\bf off-line} routing protocol that needs $\mathcal
O(C+D)$ steps to route all the requests, where $C$ is the congestion and
$D$ is the dilation of the path system.
\end{theorem}

In addition, in \cite{258658} the authors show that, given the set
of packets to be sent, it is possible to find in polynomial time a
path system with $C+D$ within a factor 4 of the optimum. Thus,
Theorem \ref{teo:C+D} can be announced in a more general way:

\begin{theorem}[\cite{258658}]
For any set of requests, there is an {\bf off-line} routing protocol that
needs $\mathcal O(C+D)$ steps to route all the requests, where $C+D$ is
the minimum $congestion+dilation$ over all the possible path systems.
\end{theorem}

Furthermore, this routing protocol uses fixed buffer size, i.e., the queue
size at each is bounded by a constant at each step. Nevertheless, it
is important to notice that a huge constant may be hidden inside the
$\mathcal O$ notation.
%V-O modified
 As we said before, this result
was further improved in \cite{Leighton_fast_95} where the same
authors gave an explicit algorithm. These algorithms to compute
the optimal schedule are centralized. In a distributed algorithm
nodes must make their decisions independently, based on the
packets they see, without the use of a centralized scheduler. In
\cite{Universal97} Ostrovsky and Rabani gave a distributed
randomized algorithm running in $\mathcal{O}(C+D +
\log^{1+\epsilon}(n))$ steps. We refer to Scheideler's
thesis~\cite{Sch98} for a complete compilation of general packet
routing algorithms.

%SHALL WE DISTINGUISH HERE BETWEEN CENTRALIZED AND DISTRIBUTED
%ALGORITHMS???

%\vspace{.3cm}

\subsubsection{On-line routing}
\label{sec:on}

In the on-line setting, the oldest on-line protocol that deviates only by
a factor logarithmic in $n$
%C$ (congestion) and $D$ (dilation),
from the best possible runtime $\mathcal O(C+D)$ for arbitrary
path-collections is the protocol presented by Leighton, Maggs and
Rao in the same paper~\cite{LMR88}, running in $\mathcal O(C+D\log
(Dn))$ steps with high probability.
%%% {\color{blue}The authors call the algorithm on-line, rather than
%%% distributed. \emph{What we mean? this section is about online algorithms ...}}
This schedule assumes that the paths are given a priori,
hence it does not consider the problem of choosing the paths to route the
packets.

The results of \cite{AAF+97} provide a routing algorithm that is
$\log n$ competitive with respect to the congestion. In other
words, it is worse than an optimal off-line algorithm only by a
factor $\log n$. In this setting the demands arrive one by one and
the algorithm routes calls based on the current congestion on the
various links in the network, so this can be achieved only via
centralized control and serializing the routing requests. In
\cite{AA94} the authors gave a distributed algorithm that
repeatedly scans the network so as to choose the routes. This
algorithm requires shared variables on the edges of the network
and hence is hard to implement. Note that the two on-line
algorithms above depend on the demands and are therefore {\it
adaptive}. Recall that an {\it oblivious} routing strategy is
specified by a path system $\mathcal P$ and a function $w$
assigning a weight to every path in $\mathcal P$. This function
$w$ has the property that for every source-destination pair
$(s,t)$, the system of flow paths $\mathcal P_{s,t}$ for $(s,t)$
fulfills $\sum_{q \in \mathcal P_{s,t}} w(q) = 1$. One can think
of this function as a frequency distribution among several paths
going from an origin $s$ to a destination $t$. In {\it adaptive}
routing, however, the path taken by a packet may also depend on
other packets or events taking place in the network during its
travel. Remark that every oblivious routing strategy is obviously
on-line and distributed.

The first paper to perform a worst case theoretical analysis on
oblivious routing is the paper of Valiant and
Brebner~\cite{ValBre81}, who considered routing on specific
network topologies such as the hypercube. They gave a randomized
oblivious routing algorithm. Borodin and Hopcroft~\cite{BoHo82}
and subsequently Kaklamanis, Krizanc, and Tsantilas~\cite{KKT91}
showed that deterministic oblivious routing algorithms cannot
approximate well the minimal load on any non-trivial network.

In a recent paper, R\"acke~\cite{Rac02} gave the construction of a
polylog competitive oblivious routing algorithm for general
undirected networks. It seems truly surprising that one can come
close to minimal congestion without any information on the current
load in the network. This result has been improved by Azar
\emph{et al.}~\cite{ACF+04}. Lower bounds on the competitive ratio
of oblivious routing have been studied for various types of
networks. For example, for the $d$-dimensional mesh, Maggs
\emph{et al.}~\cite{MHV+97} gave the $\omega(\frac{C_*}d (\log
n))$ lower bound on the competitive ratio of an oblivious
algorithm on the mesh, where $C_*$ is the optimal congestion.

So far, the oblivious algorithms studied in the literature have
focused on minimizing the congestion while ignoring the dilation.
In fact, the quality of the paths should be determined by the
congestion $C$ and the dilation $D$.  An open question is whether
$C$ and $D$ can be controlled simultaneously. An appropriate
parameter to capture how good is the dilation of a path system is
the \emph{stretch}, defined as the maximum over all packets of the
ratio between the length of the path taken by the routing protocol
and the length of a shortest path from source to destination. In a
recent work, Bush \emph{et al.}~\cite{BMX05} considered again the
case of the $d$-dimensional mesh. They presented an on-line
algorithm in which $C$ and $D$ are both within $\mathcal O(d^2)$
of the potential optimal, i.e., $D = \mathcal O(d^2 D_*)$ and $C \mathcal O(dC_* \log(n))$, where $D_*$ is the optimal dilation.
Note that by the results of Maggs \emph{et al.}~\cite{MHV+97}, it
is impossible to have a factor better than $\Omega(\frac{C_*}d
\log n )$.

There is a simple counter-example network that shows that in
general the two metrics (\emph{dilation} and \emph{congestion})
are orthogonal to each other: take an adjacent pair of nodes $u,v$
and $\Theta (\sqrt n)$ disjoint paths of length $\Theta(\sqrt n)$
between $u$ and $v$. For packets traveling from $u$ to $v$, any
routing algorithm that minimizes congestion has to use all the
paths, however, in this way some packets follow long paths, giving
high stretch. Nevertheless, in grids~\cite{BMX05}, and in some
special kind of geometric networks~\cite{BMXb05} the congestion is
within a poly-logarithmic factor from optimal and stretch is
constant ($d$ the dimension). As mentioned before an interesting
open problem is to find other classes of networks where the
congestion and stretch are minimized simultaneously \cite{ABD+06}.
Possible candidates for such networks could be for example
bounded-growth networks, or networks whose nodes are uniformly
distributed in closed polygons, which describe interesting cases
of wireless networks.

The recent paper of Maggs \cite{Mag06} surveys a collection of
theoretical results that relate the congestion and dilation of the
paths taken by a set of packets in a network to the time required
for their delivery.

%({\bf R\'ef\'erence \`a \cite{ABD+06} (le fichier open.pdf)}

%\textbf{Il faut aussi regarder ce papier de survey tr\`es r\'ecent de
%Maggs {\it A Survey of Congestion+Dilation Results for Packet Scheduling}
%\cite{Mag06}, j'arrive pas \`a le trouver sur internet!}

%\newpage
\subsection{Routing Problems}
\label{sec:problem}
%If in the general problem the input is any network with any set of
%packets, the routing problem can be specialized by restricting the types
%of the inputs.

The initial and final positioning of the packets has a direct influence on
the time needed for their routing. Considering static packet
configuration, the most studied constraints refer to the maximum number of
packets that a node can send and receive. Due to
their practical importance, some of these problems have specific names:\\

\begin{itemize}
\item[1.] \textbf{Permutation routing}: each node is
the origin and the destination of at most one packet. To measure the
routing capability of an interconnection network, the partial permutation
routing (PPR) problem is usually used as the metric.
\item[2.] \textbf{$(\ell,k)$-routing:} each node is the origin of at most $\ell$
packets and destination of at most $k$ packets. Permutation routing
corresponds to the case $\ell=k=1$ of $(\ell,k)$-routing. Another
important particular case is the \textbf{$(1,k)$-routing}, in which each
node sends at most one packet and receives at most $k$ packets.
\item[3.] \textbf{$(1,any)$-routing}: each node is the
origin of at most one packet but there are no constraints on the number of
packets that a node can receive.
\item[4.] \textbf{$r$-central routing}: all nodes at distance at most $r$ of a central
node send one message to this central node.\\
\end{itemize}

\noindent In all these problems, we are given an initial packet
configuration and the objective is to route all packets to their
respective destinations minimizing the total routing time, under the
constraint that each edge can be used by at most one packet at the same
time.

Besides of the constraints about the initial and final positions of the
packets, there also exist different routing models at the intermediate
nodes of the network. For instance, in the \emph{hot potato model} no
packet can be stored at the nodes of the network, whereas in the
\emph{store-and-forward} at each step a packet can either stay at a node
or move to an adjacent node.

On the other hand, one can consider constraints on the number of
incident edges that each node of the network can use to send or
receive packets at the same time. In the $\Delta$-port model
\cite{deltaport}, each node can send or receive packets through
all its incident edges at the same time.

In this article we study the store-and-forward $\Delta$-port model. In
addition, we suppose that cohabitation of multiple packets at the same
node is allowed. I.e., a queue is required for each outgoing edge at each
node.

The nature of the links of the network is another factor that influences
the routing efficiency. The type of links is usually one of the following:
full-duplex or half-duplex. In the full-duplex case there are two links
between two adjacent nodes, one in each direction. Hence two packets can
transit, one in each direction, simultaneously. In the half-duplex case
only one packet can transit between two nodes, either in one direction of
the edge or in the other. In this paper we study both half and
full-duplex links.

\subsection{Topologies}
\label{sec:previous}

We now give a brief summary of various cases of $(\ell,k)$-routing
and $(1,any)$-routing that have been studied for several specific
topologies. More precisely, in Section \ref{sec:different} we list
some of the important results for some networks which have
attracted interest in the literature, like hypercubes and
circulant graphs. We move then to plane grids in Section
\ref{sec:plane}. It is well known that there exist only three
possible tessellations of the plane into regular polygons
\cite{williams1979gfn} : squares, triangles and hexagons. These
graphs are those which we study in this article.

%Some rough additional notes on the state-of-the-art of this problem can be
%found in Appendix \ref{sec:app_biblio}.

\subsubsection{Different network topologies}
\label{sec:different}

Hwang, Yao, and Dasgupta~\cite{low_hyper} studied the permutation
routing problem in low-dimensional hypercubes ($d\leq 12$). They
gave optimal or good-in-the-worst-case oblivious algorithms.
%, i.e. algorithms in which the path used by
%a packet is entirely determined by its origin and its destination.
 Another network widely studied in the literature is the two dimensional mesh
with row and column buses. This network can also be diversified
according to the capacities of the buses. In \cite{Suel_94}, Suel
gave a deterministic algorithm to solve the permutation routing
problem in such networks. The algorithm provides a schedule using
at most $n+o(n)$ steps and queues of size two.
%V-O removed: we have already used the queue several times before.
%%%% where the queue is the maximum number of packets that have to be
%stored at an intermediate node.
He also proposed a deterministic algorithm for $r$-dimensional
arrays with buses working in $(2-\frac{1}{r})n+o(n)$ steps and
still using queues of size 2. In \cite{KNR91}, the authors studied
the $(\ell,\ell)$-routing problem in the mesh grid with two
diagonals and gave a deterministic algorithm using $\frac{2\ell
n}{9}+\mathcal(\ell n^{2/3})$ steps for $\ell \geq 9$.

%V-O modified
In \cite{big_foot}, the authors introduced an algorithm called
\emph{big foot} algorithm. The idea of this algorithm is to
identify two types of links and to move towards the destination
using first the links of the first type and then those of the
second type. The algorithms we develop will use such a strategy.
They give an optimal centralized algorithm for the permutation
routing problem in full-duplex 2-circulant graphs and
\emph{double-loop} networks. This later network is of great
practical importance. It is modeled by a graph with vertex set
%V-O modified
$V=\{v_0, \dots, v_{n-1}\}$ such that there are two integers $h_1$
and $h_2$ such that the edge set is $E=\{v_iv_{i \pm h_1},v_iv_{i
\pm h_2}\}$. The permutation routing problem in this network is
studied by Dobravec, Robi\v c, and  \v Zerovnik~\cite{Janez3}. The
authors gave an algorithm for the permutation routing problem
which in mean uses $1.12 \ell$ steps (the mean being empirically
measured). In \cite{Janez2} the authors described an optimal
centralized permutation routing algorithm in $k$-circulant graphs
($k\geq 2$), and in \cite{Janez1} an optimal distributed
permutation routing in 2-circulant graphs was obtained.

%V-O modified
The problem has been also studied for packets arriving dynamically
by Havil in~\cite{havill01online}, where an optimal online
schedule for the linear array is given. Havil also gave a
2-approximation for rings and show that, using shortest path
routing, no better approximation algorithm exists. Jan and
Lin~\cite{CCC} studied Cube Connected Cycles $CCC(n,2^n)$. These
are hypercubes of dimension $n$ where each node is replaced by a
cycle of length $n$. They gave an algorithm working in
$\mathcal{O}(n^2)$ with $\mathcal{O}(1)$ buffers for the online
partial permutation routing (PPR).

\subsubsection{Plane grids}
\label{sec:plane}

Maybe the most studied networks in the literature are the two dimensional
grids (or plane grids), and among them in particular the square grid has
deserved special attention. Let us briefly overview what has been
previously done on $(\ell,k)$-routing in plane grids.

%V-O modified
Leighton, Makedon, and Tollis~\cite{constant95} obtained the first
optimal permutation routing with running time $2n-2$ and queues of
size 1008. Rajasekaran and Overholt~\cite{rajasekaran} reduced the
queue size to 112. Sibeyn, Chlebus, and
Kaufmann~\cite{sibeyn97deterministic} reduced this to 81.
Furthermore, they provided another algorithm running in
near-optimal time $2n + \mathcal{O}(1)$ steps with a maximum queue
size of only 12. Makedon and Symvonis~\cite{Makedon_opt_93}
introduced the $(1,k)$-routing and the $(1,any)$-routing problems
and gave an \emph{asymptotically} optimal algorithm for
$(1,k)$-routing on plane grids, with queues of small constant
size.  This result was further improved by Sibeyn and Kaufman
in~\cite{Sibeyn_1k}, where they gave a near-optimal deterministic
algorithm running in $\sqrt{k}\frac{n}{2}+\mathcal{O}(n)$ steps.
They gave another algorithm, slightly worse, in terms of number of
steps, but with queues of size only 3. They also studied the
general problem of $(\ell,k)$-routing in square grids. They
proposed lower bounds and near-optimal randomized and
deterministic algorithms. They finally extended these algorithms
to higher dimensional meshes. They performed $(\ell,\ell)$-routing
in $\mathcal{O}(\ell n)$ steps, the lower bound being
$\Omega(\sqrt{\ell kn})$ for $(\ell,k)$-routing. Finally,
Pietracaprina and Pucci~\cite{pietracaprina01optimal} gave
deterministic and randomized algorithms for $(\ell,k)$-routing in
square grids, with constant queue size. The running time is
$\mathcal{O}(\sqrt{\ell kn})$ steps, which is optimal according to
the bound of \cite{Sibeyn_1k}. This work closed a gap in the
literature, since optimal algorithms were only known for $\ell =1$
and $\ell = k$.

%\subsubsection{Hexagonal and Triangular Grids}

%In this paper we mainly focus on hexagonal grids because the tessellation
%of the plane with hexagons may be considered as the most natural way to
%model radio telecommunication networks. Indeed cells have a good diameter
%to area ratio.
%If centers of neighboring cells are connected, we obtain a
%triangular grid. \textsc{Fig. } \ref{fig:hex_net} represents an
%interconnection network of base stations using wireless connections
%\cite{Stoj1}

Nodes in a hexagonal network are placed at the vertices of a
regular triangular tessellation, so that each node has up to six
neighbors. In other words, a hexagonal network is a finite
subgraph of the triangular grid. These networks have been studied
in a variety of contexts, specially in wireless and
interconnection networks. The most known application may be to
model cellular networks with hexagonal networks where nodes are
base stations. But these networks have been also applied in
chemistry to model benzenoid hydrocarbons
\cite{chemistry,klavzar}, in image processing and computer
graphics \cite{comp_graph}.

In a radiocommunication wireless environment \cite{Stoj1}, the
interconnection network among base stations constitutes a
hexagonal network, i.e., a triangular grid,
 as it is shown in \textsc{Fig.} \ref{fig:hex_net}.

\begin{figure}[t]
\begin{center}
\includegraphics[width=5.0cm]{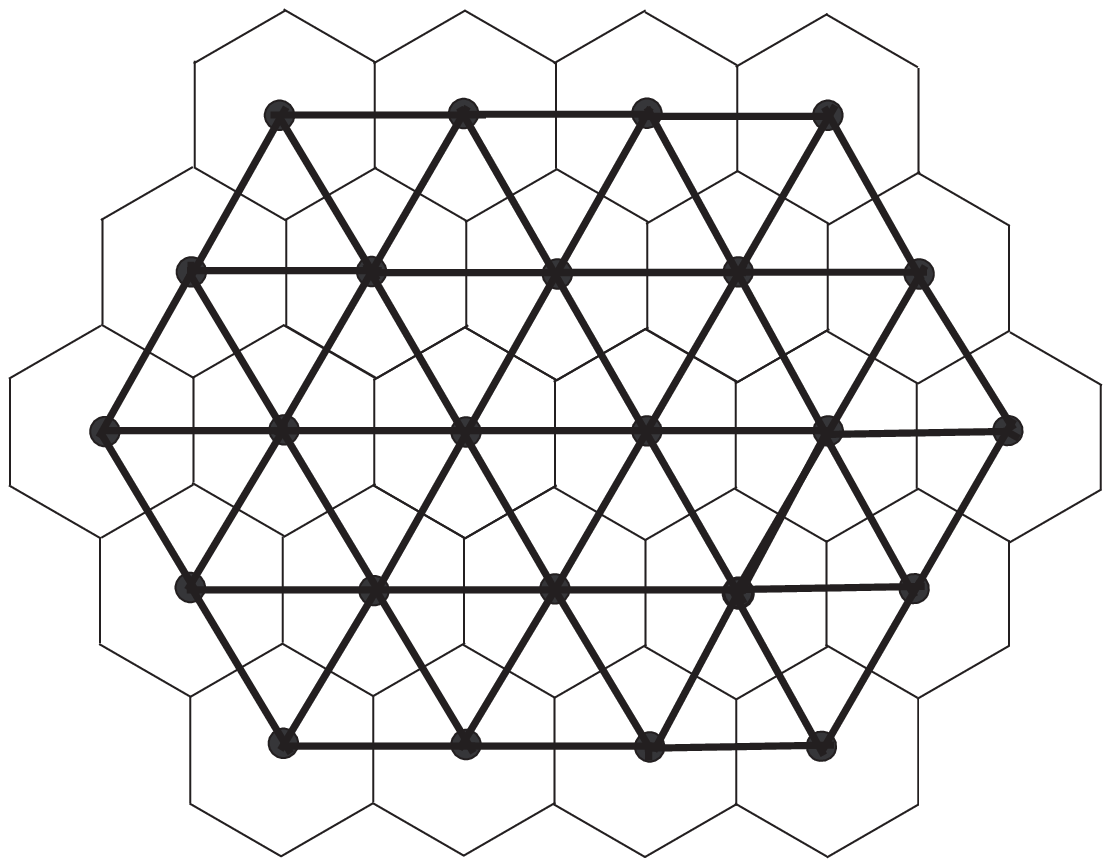}
\caption{Hexagonal network ($\triangle$) and hexagonal tessellation ({\Large\varhexagon}).}
\label{fig:hex_net}
\end{center}
\end{figure}

%V-O: modified
Tessellation of the plane with hexagons may be considered as the most
natural model of networks because cells have optimal diameter to area ratio.
%Hexagonal
%networks are finite subgraphs of the triangular grid.
The triangular grid can also be obtained from the basic 4-mesh by
adding NE to SW edges, which is called a 6-mesh in \cite{trobec}.
Here we study convex subgraphs, i.e., subgraphs that contain all
shortest paths between all pairs of nodes, of the square,
triangular and hexagonal grids. Summarizing, to the best of our
knowledge the only optimal algorithms concerning
$(\ell,k)$-routing on plane grids (according to the lower bound of
\cite{Sibeyn_1k}) have been found on square grids, but modulo a
constant factor \cite{pietracaprina01optimal}. On triangular and
hexagonal grids, the best results are randomized algorithms with
good performance \cite{Sybeyn97}.

%\newpage
\subsection{Our Contribution}
%In this paper we give optimal routing algorithm for the $(1,1)$-routing
%problem in hexagonal network, $\ell$-central routing problem in the
%square, triangular and hexagonal grids. We also give good approximation
%algorithm for the $(\ell,k)$-routing problem.
In this paper we study the permutation routing, $r$-central and
$(\ell,k)$-routing problems on plane grids, that is square grids,
triangular grids and hexagonal grids. We use the \emph{store-and-forward}
$\Delta$-port model, and we consider both full and half-duplex networks.

%V-O modified
We have seen in Section \ref{sec:plane} that the only plane grid for which
there existed an optimal $(\ell,k)$-routing is the square grid. In
addition, these articles
concerning $(\ell,k)$-routing in plane grids are optimal modulo a constant
factor. In this paper we improve these results by giving tight algorithms
including the constant factor, in the cases of square, triangular and hexagonal
grids. It is important to stress that all the algorithms presented in this
paper except the one given in
\ref{sec:weighted} are distributed.
Our algorithms only use shortest paths, therefore they achieve minimum
stretch. In addition, the algorithms are oblivious, so they can be used in
an on-line scenario. However the performance guarantees that we prove apply
only to the off-line case. The new results are the
following:
%V-O modified
\begin{itemize}

\item[1.] Tight (also including the constant factor) permutation routing algorithms in
full-duplex hexagonal grids, and half duplex triangular and hexagonal
grids.
\item[2.] Tight (also including the constant factor) $r$-central routing algorithms in
 triangular and hexagonal grids.

 \item[3.] Tight (also including the constant factor) $(k,k)$-routing algorithms in
square, triangular and hexagonal grids.

\item[4.] Good approximation algorithms for $(\ell,k)$-routing in
square, triangular and hexagonal grids.\\
\end{itemize}

\noindent This paper is structured as follows. In Section
\ref{sec:perm_rout} we study the permutation routing problem.
Although permutation routing had already been solved for square
grids, we begin in Section \ref{sec:square} by illustrating our
algorithm for such grids. Then in Section \ref{sec:tri} we give
tight permutation routing algorithm for half-duplex triangular
grids, using the optimal algorithm of \cite{INOC07}. In Section
\ref{sec:hex} we provide a tight permutation routing algorithm for
full-duplex hexagonal grids and a tight permutation routing
algorithm for half-duplex hexagonal grids. In Section
\ref{sec:any} we focus on $(1,any)$-routing, giving an optimal
$r$-central routing algorithms for the three types of grids. We
finally move in Section \ref{sec:lk} to the general
$(\ell,k)$-routing problem. We provide a distributed algorithm for
$(\ell,k)$-routing in any grid, using the ideas of the optimal
algorithm for permutation routing. We also prove lower bounds for
the worst-case running time of any algorithm using shortest path
routing.
%The approximation guarantee is proved to be within a small
%constant of the optimal.
In addition, these lower bounds allow us to prove that our algorithm turns
out to be tight when $\ell=k$, yielding in this way a tight
$(k,k)$-routing algorithm in square, triangular and hexagonal grids. We
propose in~\ref{sec:weighted} an approach to
$(\ell,k)$-routing in terms of a graph coloring problem: the
\textsc{Weighted Bipartite Edge Coloring}. We give a centralized algorithm
using this reduction.

%\newpage
\section{Permutation Routing}
\label{sec:perm_rout} As we have already said in Section
\ref{sec:intro}, in the permutation routing problem, each
processor is the origin of at most one packet and the destination
of no more than one packet. The goal is to minimize the number of
time steps required to route all packets to their respective
destinations. It corresponds to the case $\ell=k=1$ of the general
$(\ell,k)$-routing problem. This problem has been studied in a
wide diversity of scenarios, such as Mobile Ad Hoc Networks
\cite{init_prot}, Cube-Connected Cycle (CCC) Networks \cite{CCC},
Wireless and Radio Networks \cite{wireless}, All-Optical Networks
\cite{optical} and Reconfigurable Meshes \cite{mesh_reconf}.

%This problem has been widely studied in the literature for several types
%of networks
%\cite{init_prot,CCC,wireless,optical,mesh_reconf,big_foot,Janez3,opt_hexnet_preprint}...

%BLAH, BLAH, BLAH...

%V-O modified
In a grid with full-duplex links an edge can be crossed simultaneously by
two messages, one in each direction. Equivalently, each edge between two
nodes $u$ and $v$ is made of two independent arcs $uv$ and $vu$,
as illustrated in Fig. \ref{fig:2arcs_axis}a.

\begin{figure}[h!tb]
\begin{center}
\includegraphics[width=9.6cm]{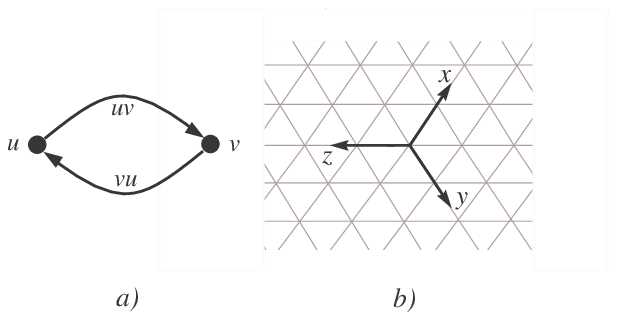}
\caption{\textbf{a)} Each edge consists of two independent links. \textbf{b)} Axis used in a
triangular grid.} \label{fig:2arcs_axis}
\end{center}
\end{figure}

\begin{observation}
\label{obs:2}If the network is half-duplex, it is easy to
construct a 2-approximation algorithm from an optimal algorithm
for the full-duplex case by introducing \emph{odd-even} steps, as
explained for example~in~\cite{Janez3}.
\end{observation}

\subsection{Square grid}
\label{sec:square}

Many communication networks are represented by graphs satisfying the
following property: for any pair of nodes $u$ and $v$, the edges of a
shortest path from $u$ to $v$ can be partitioned into $k$ disjoint classes
according to a well-defined criterium. For instance, on a triangular grid
 the edges of a shortest path can be partitioned into
\emph{positive} and \emph{negative} ones \cite{INOC07}. Similarly,
on a $k$-circulant graph the edges can be partitioned into $k$
classes according to their length.

In graphs that satisfy this property there exists a natural
routing algorithm: route all packets along one class of edges
after another. For hexagonal networks this algorithm turns out to
be optimal \cite{INOC07}. Optimality for 2-circulant graphs is
proved using a static approach in \cite{big_foot}, and recently
using a dynamic distributed algorithm in \cite{Janez1}. In
\cite{big_foot} the authors introduce the notion of
\emph{big-foot} algorithms because their algorithm routes packets
first along \emph{long hops} and then along \emph{short hops} in a
2-circulant graph.

On the square grid, the \emph{big-foot} algorithm consists of two phases,
moving each packet first horizontally and then vertically. In this way a
packet may wait only during the second phase. Using the fact that all destinations
are distinct, the optimality for square grid is easy to prove.
Summarizing, it can be proved that

\begin{theorem}
\label{teo:perm_square} There is a translation invariant oblivious optimal
permutation routing algorithm for full-duplex networks that are convex
subgraphs of the infinite square grid.
\end{theorem}

\subsubsection{Regarding the queue size}

Of course, this is not the first optimal permutation routing
result on square grids, as the classical $ x-y$ routing (first
route packets through the horizontal axis, and then through the
vertical axis) has been used for a long time. Thus, another more
challenging issue is to reduce the queue size, as we have already
discussed in Section \ref{sec:plane}. Leighton describes in
\cite{duplex} a simple off-line algorithm for solving any
permutation routing problem in $3n-3$ steps on a $n \times n$
square grid, using queues of size one. Since the diameter of a $n
\times n$ square grid is $2 n -2$, this algorithm provides a
$\frac{3}{2}$-approximation. The main drawback is that this
algorithm is off-line and centralized. In contrast, our oblivious
distributed algorithm is optimal in terms of running time, but it
is easy to see that on a $n \times n$ square grid, the queue size
can be $\frac{n-1}{2}$. Up to date, the best algorithm running in
optimal time to route permutation routing instances on square
grids is the algorithm of Sibeyn et al.
\cite{sibeyn97deterministic}, using queues of size 81. So far,
there is no algorithm that guarantees optimal running time with
queues of size 1, and it is unlikely that such an algorithm
exists.

\begin{observation} The same observation regarding the unbounded queue size
applies to all the algorithms described in this article. However, our aim
is to match the optimal running time, rather than minimizing the queue
size. Additionally, it turns out that some appropriate modifications of
the permutation routing algorithms that we provide for plane grids allow
us to find oblivious algorithms which route any permutation within a
factor 3 of the optimal running time, and using queues of size 1 (in fact,
we can say something stronger: we just need memory to keep 1 message at
each node). We do not describe these modifications in this article.
\end{observation}

\subsection{Triangular grid}

\label{sec:tri} We use the addressing scheme introduced in
\cite{Stoj1} and used also in \cite{INOC07}: we represent any
address on a basis consisting of three unitary vectors
\emph{\textbf{i}, \textbf{j}, \textbf{k}} on the directions of
three axis $x,y,z$ with a $120$ degree angle among them,
intersecting on an arbitrary (but fixed) node $O$ . This node is
the origin and is given the address
$\text{\textbf{\emph{O}}}=(0,0,0)$. This basis is represented in
\textsc{Fig.} \ref{fig:2arcs_axis}b. Thus, we can assume that each
node $P\in V$ is labeled with an address
$\text{\textbf{\emph{P}}}=(P_1,P_2,P_3)$ expressed in this basis
\{\emph{\textbf{i}, \textbf{j}, \textbf{k}}\} with respect to the
origin $O$. At the beginning, each node $S$ knows the address of
the destination node $D$ of the message placed initially at $S$,
and computes the relative address
$\overrightarrow{SD}=\text{\textbf{\emph{D}}} -
\text{\textbf{\emph{S}}}$ of the message. Note that this relative
address does not depend on the choice of the origin node $O$. This
relative address is the only information that is added in the
heading of the message to be transmitted, constituting in this way
the packet to be sent through the network.

Using that
$\text{\emph{\textbf{i}}}+\text{\emph{\textbf{j}}}+\text{\emph{\textbf{k}}}=0$,
it is easy to see that if $(a,b,c)$ and $(a',b',c')$ are the relative
addresses of two packets, then $(a,b,c)=(a',b',c')$ if and only if there
exists $d \in \mathbb{Z}$ such that $a'=a+d$, $b'=b+d$, and $c'=c+d$.

We say that an address $\overrightarrow{SD}=(a,b,c)$ is of the
\emph{shortest path form} if there is a path from node $S$ to node $D$,
consisting of $a$ units of vector $\text{\textbf{\emph{i}}}$, $b$ units of
vector $\text{\textbf{\emph{j}}}$ and $c$ units of vector
$\text{\textbf{\emph{k}}}$, and this path has the shortest length.

%V-O modified any ...> the two other
\begin{theorem}[\cite{Stoj1}]
\label{teo:short} An address $(a,b,c)$ is of the shortest path form if and
only if at least one component is zero, and the two other components do not have
the same sign.
\end{theorem}
\begin{corollary}[\cite{Stoj1}]
Any address has a unique shortest path form.
\end{corollary}
Thus, each address $\overrightarrow{SD}$ written in the shortest path form
has at most two non-zero components, and they have different signs. In
fact, it is easy to find the shortest path form using the next result.
\begin{theorem}[\cite{Stoj1}]
If
$\overrightarrow{SD}=a\text{\textbf{\emph{i}}}+b\text{\textbf{\emph{j}}}+c\text{\textbf{\emph{k}}}$,
then
$$
|\overrightarrow{SD}|=\min(|a-c|+|b-c|,|a-b|+|b-c|,|a-b|+|a-c|).
$$
\end{theorem}

Permutation routing on full-duplex triangular grids has been
solved recently \cite{INOC07} attaining the distance lower bound
of $\ell_{\max}$ routing steps, where $\ell_{\max}$ is the maximum
length over the shortest paths of all packets to be sent through
the network.

As said in Remark \ref{obs:2}, if the network is half-duplex, one can
construct a 2-approximation algorithm from an optimal algorithm for the
full-duplex case by introducing \emph{odd-even} steps. Thus, using this
algorithm we obtain an upper bound of $2\ell_{\max}$ for half-duplex
triangular grids.

Let us show with an example that this na\"{i}ve algorithm is tight. That is,
we shall give an instance requiring at least $2\ell_{\max}$ running steps,
implying that no better algorithm for a general instance exists. Indeed,
consider a set of nodes distributed along a line on the triangular grid.
We fix $\ell_{\max}$ and an edge $e$ on this line, and put $\ell_{\max}$
packets at each side of $e$ along the line, at distance at most
$\ell_{\max}-1$ from an end-vertex of $e$. For each packet, each
destination is chosen on the other side of $e$ with respect to its origin,
at distance exactly $\ell_{\max}$ from the origin. It is easy to check that
the congestion of $e$ (that is, the number of shortest paths containing
$e$) is $2\ell_{\max}$, and thus any algorithm using shortest path routing
cannot perform in less than $2\ell_{\max}$ steps. On the other hand,
$\ell_{\max}$ is a lower bound for any distance, yielding that the
approximation ratio of our algorithm is at most $2$.

Previous observations allow us to state the next result:

\begin{theorem}
There exists a tight permutation routing algorithm for
half-duplex triangular grids performing in at most $2\ell_{\max}$ steps,
where $\ell_{\max}$ is the maximum length over the shortest paths
of all packets to be sent. This algorithm is a 2-approximation algorithm for a general instance.
\end{theorem}

%BISECTION BOUND?

\subsection{Hexagonal grid}
\label{sec:hex} In a hexagonal grid one can define three types of
\emph{zigzag chains} \cite{honeycomb}, represented with thick
lines in \textsc{Fig.} \ref{fig:axis_hex}. Similarly to the
triangular grid, in the hexagonal grid any shortest path between
two nodes uses at most two types of zigzag chains
\cite{honeycomb}. Let us now give a lower bound for the running
time of any algorithm. Consider the edge labeled as $e$ in
\textsc{Fig.} \ref{fig:axis_hex}, and the two chains containing it
(those shaping an $X$). Fix $\ell_{\max}$ and $e$, and put one
message on all nodes placed at both chains at distance at most
$\ell_{\max}-1$ from an endvertex of $e$. As in the case of the
triangular grid, choose the destinations to be placed on the other
side of $e$ along the same zigzag chain than the originating node,
at distance exactly $\ell_{\max}$ from it. It is clear that all
the shortest paths contain $e$. It is also easy to check that the
congestion of $e$ is  $4\ell_{\max}-4$ in this case, constituted
of symmetric loads $2\ell_{\max}-2$ in each direction of $e$.
Thus, $2\ell_{\max}-2$ establishes a lower bound for the running
time of any algorithm in the full-duplex case, whereas
$4\ell_{\max}-4$ is a lower bound for the half-duplex case, under
the assumption of shortest path routing.

\begin{figure}[h!tb]
\begin{center}
\includegraphics[width=5.0cm]{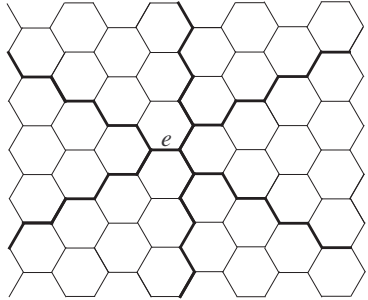}
\caption{Zigzag chains in a hexagonal grid.} \label{fig:axis_hex}
\end{center}
\end{figure}

%\begin{observation}
%Note that this lower bound has been obtained assuming that packets
%are always routed through a shortest path. Without this
%assumption, another analysis should be carried out to get the
%bound.
%\end{observation}

Let us now describe a routing algorithm which reaches this bound. We have
three types of edges according to the angle that they form with any fixed
edge. Each edge belongs to exactly two different chains, and conversely
each chain is made of two types of edges. Moreover, in an infinite
hexagonal grid any two chains of different type intersect exactly on one
edge.

Given a pair of origin and destination nodes $S$ and $D$, it is
possible to express the relative address $D-S$ counting the number
of steps used by a shortest paths on each type of chain. In this
way we obtain an address $D-S=(a,b,c)$ on a generating system made
of unitary vectors following the directions of the three types of
chains (it is not a basis in the strict sense, since these vectors
are not linearly independent on the plane. However, we will call
it so). Choose this basis so that the three vectors form angles of
$120$ degrees among them. As it happens on the triangular grid
\cite{Stoj1,opt_hexnet_preprint}, there are at most two non-zero
components (see \cite{honeycomb}), and in that case they must have
different sign. Nevertheless, in this case, the address is not
unique, since an edge placed at the \emph{bent} (that is, a change
from a type of chain to another) of a shortest path is part of
both types of chains. Anyway, this ambiguity is not a problem in
the algorithm we propose, as we will see below.

Suppose first that edges are bidirectional or, said otherwise,
full-duplex. Roughly, the idea is to use the optimal algorithm for
triangular grids described in \cite{opt_hexnet_preprint}, and
adapt it to hexagonal grids. For that purpose we label the three
types of zigzag chains $c_1,c_2,c_3$, and the three types of edges
$e_1,e_2,e_3$. Without loss of generality, we label them in such a
way that $c_1$ uses edges of type $e_2$ and $e_3$, $c_2$ uses
$e_1$ and $e_3$, and $c_3$ uses $e_1$ and $e_2$ (see \textsc{Fig.}
\ref{fig:chains_edges}).

\begin{figure}[h!tb]
\begin{center}
\includegraphics[width=9cm]{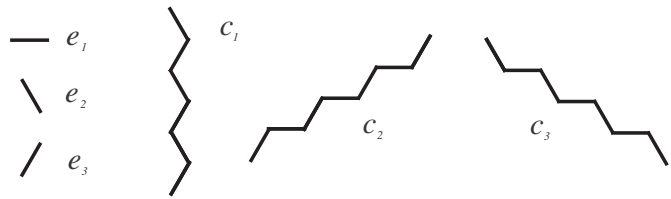}
\caption{3 types of chains and edges in a hexagonal grid.} \label{fig:chains_edges}
\end{center}
\end{figure}

For each type of edge, we define two phases according to the type of chain
that uses this type of edge. This defines two global phases, namely:
during Phase 1, $c_1$ uses $e_2$, $c_2$ uses $e_3$, and $c_3$ uses $e_1$.
Conversely, during Phase 2 $c_1$ uses $e_3$, $c_2$ uses $e_1$, and $c_3$
uses $e_2$.

We suppose that at each node packets are grouped into distinct queues
according to the next edge (according to the rules of the algorithm) along
its shortest path. Given the relative addresses $D-S$ in the form
$(a,b,c)$, the algorithm
can be described as follows. \\

%\newpage

At each node $u$ of the network:

\begin{itemize}
\item[\textbf{1)}] During the first step, move all packets along the
direction of their negative component. If a packet's address has
only a positive component, move it along this direction.
\item[\textbf{2)}] From now on, change alternatively between Phase 1 and Phase 2. At each step (the same for both phases):
\begin{itemize}
\item[\textbf{a)}] If there are packets with negative components, send them immediately along the direction of this component.
\item[\textbf{b)}] If not, for each outgoing edge order the packets
in decreasing number of remaining steps, and send the first packet of each
queue.
\end{itemize}
\item[\textbf{3)}] If at some point, all the packets in $u$ have
remaining distance one, send them immediately.\\
%At the end of the $(2\ell_{\max}-3)th$ step, move all packets along their unique non-zero component.\\
\end{itemize}

Let us analyze the correctness and optimality of this algorithm.

In \textbf{1)} all packets can move, since initially there is at
most one packet at each node. In \textbf{2a)}, there can only be
one packet with negative component at each outgoing edge
\cite{opt_hexnet_preprint}. In \textbf{2b)} \emph{the} packet with
maximum remaining length at each outgoing edge is unique. Indeed
all these packets are moving along their last direction (their
negative component is already finished, otherwise they would be in
\textbf{2a)}) and each node is the destination of at most one
packet. Hence, using this algorithm, every $2$ steps (one of Phase
1 and one of Phase 2) the maximum remaining distance over all
packets decreases by one. Moreover, during the first step all
packets decrease their remaining distance by one. Because of this,
after the $(2\ell_{\max}-3)th$ step
%in \textbf{3)}
the maximum remaining distance has decreased at least
$1+\frac{2\ell_{\max}-4}{2}=\ell_{\max}-1$ times, hence the maximum
remaining distance is 1, and we are in \textbf{3)}. Since all destinations
are different, all packets can reach simultaneously their destinations.
Thus, the total running time is at most $2\ell_{\max}-3+1=2\ell_{\max}-2$,
meeting the \emph{worst case} lower bound. Again, $\ell_{\max}$ is a lower
bound for any instance, hence the algorithm constitutes a 2-approximation
for a general instance.
%\\

%\textbf{Written like this, all nodes need to know $\ell_{\max}$ ->
%problem. Modification: at any step, if all packets have remaining
%distance 1, send them immediately, regardless of Phase 1/2.\\}

\begin{theorem}
\label{theo:hexfull} There exists a tight permutation routing
algorithm for full-duplex hexagonal grids performing in
$2\ell_{\max}-2$ steps, where $\ell_{\max}$ is the maximum length
over the shortest paths of all packets to be sent.
\end{theorem}

\begin{observation}
The optimality stated in Theorem \ref{theo:hexfull} is true only under the
assumption of shortest path routing. This means that for certain traffic
instances the total deliver time may be shorter if some packets do not go
through their shortest path. To illustrate this phenomenon, consider the
example of \textsc{Fig.} \ref{fig:not_shortest}. Node labeled $i$ wants to
send a message to node labeled $i'$, for $i=1,\ldots,8$. We have that
$\ell_{\max}=5$, and thus our algorithm performs in $2\ell_{\max}-2=8$
steps. It is clear that all shortest paths use edge $e$, and its
congestion bottlenecks the running time. Suppose now that we route the
messages originating at even nodes through the path defined by the edges
$\{a,b,c,d\}$, instead of $\{f,e\}$, and keep the shortest path routing for
messages originating at odd nodes. One can check that with this routing
only $7$ steps are required.

\begin{figure}[h!tb]
\begin{center}
\includegraphics[width=6cm]{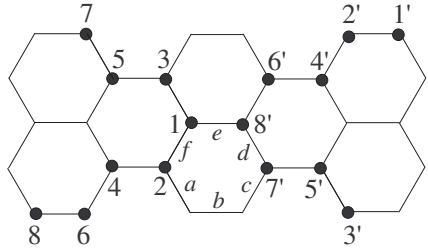}
\caption{Shortest path is not always the best choice.} \label{fig:not_shortest}
\end{center}
\end{figure}
\end{observation}

In the half-duplex case, just introduce again odd-even steps in both
phases. Thus, we have Phase 1-even, Phase 1-odd, Phase 2-even, and Phase
2-odd, which take place sequentially. Now, \textbf{1)} consists obviously
of two steps (even/odd). Using this algorithm, every $4$ steps the maximum
remaining distance decreases by one. In addition, during the first $2$
steps and during the last $2$ steps all packets decrease their remaining
distance by one. Thus, the total running time is at most twice the time of
the full-duplex case, that is $2(2\ell_{\max}-2)=4\ell_{\max}-4$ steps,
meeting again the lower bound for the running time of any routing
algorithm using shortest path routing. Again, this algorithm constitutes a
4-approximation for a general instance.

\begin{theorem}
\label{theo:hexhalf} There exists an tight permutation routing algorithm
for half-duplex hexagonal grids performing in $4\ell_{\max}-4$ steps, where
$\ell_{\max}$ is the maximum length over the shortest paths of all packets
to be sent.
\end{theorem}

\begin{observation}
As explained in  \ref{sec:embed}, there exists an embedding of the
triangular grid into the hexagonal grid with load, dilation, and
congestion $2$. Using this embedding, any algorithm performing on $k$
steps on the triangular gird performs on $2k$ steps on the hexagonal grid.
Using this fact, we obtain a permutation routing algorithm on full-duplex
hexagonal grids performing on $2\ell_{\max}$ steps. Note that the optimal
result given in Theorem \ref{theo:hexfull} is slightly better.

The same applies to half-duplex hexagonal networks, with a running time of
$4\ell_{\max}$ using the embedding, in comparison to $4\ell_{\max}-4$ steps
given by Theorem \ref{theo:hexhalf}.
\end{observation}

%\newpage
\section{$(1,any)$-Routing}
\label{sec:any} In this case the routing model is the following:
each packet has at most one packet to send, but there are no
constraints on the destination. That is, in the worst case all
packets can be sent to one node. This special case where all
packets want to send a message to the same node in often called
\emph{gathering} in the literature \cite{gathering_radio}. Notice
that this routing model is conceptually different from the
$(1,k)$-routing, where the maximum number of packets that a node
can receive is fixed \emph{a priori}.

\paragraph{Square grid}
Assume first that edges are bidirectional. The modifications for the
half-duplex case are similar to those explained in the previous section.

We will focus on the case where all packets surrounding a given vertex
want to send a packet to that vertex. We call this situation
\emph{central} routing, and if we want to specify that all nodes at
distance at most $r$ from the center want to send a packet, we note it as
\emph{$r$-central} routing. Note that this situation is realistic in many
practical applications, since the central vertex can play the role of a
router or a gateway in a local network.

\begin{lemma}[Lower Bound]
The number of steps required in a $r$-central routing is at least ${r+1
\choose 2}$.
\end{lemma}
\begin{proof}
Let us use the bisection bound \cite{survey_routing98} to prove
the result. It is easy to count the number of points at distance
at most $r$ from the center, which is $4{r+1 \choose 2}$. Now
consider the cut consisting of the four edges outgoing from the
central vertex. All packets must traverse one of these edges to
arrive to the central vertex. This cut gives the bisection bound
of $4{r+1 \choose 2}/4$ routing steps.
\end{proof}
Let us now describe an algorithm meeting the lower bound.
\begin{proposition}
\label{prop:optcentral} There exists an optimal $r$-central routing
algorithm on square grids performing in ${r+1 \choose 2}$ routing steps.
\end{proposition}

\begin{proof}
Express each node address in terms of the relative address with respect to
the central vertex. In this way each node is given a label $(a,b)$. Then,
for each packet placed in a node with label $(a,b)$ our routing algorithm
performs the following:
\begin{itemize}
\item[$\bullet$] If $ab=0$, send the packet along the direction of the
non-zero component.
\item[$\bullet$] If $ab>0$, send the packet along the vertical axis.
\item[$\bullet$] If $ab<0$, send the packet along the horizontal axis.
\end{itemize}
Queues are managed so that the packets having greater remaining distance
have priority.
%following their first direction of movement.???

This routing divides the square grid into 4 subregions surrounding the
central vertex, as shown in \textsc{Fig.} \ref{fig:squarediv}. The type of
routing performed in each subregion is symbolized by an arrow.

\begin{figure}[h!tb]
\begin{center}
\includegraphics[width=4.5cm]{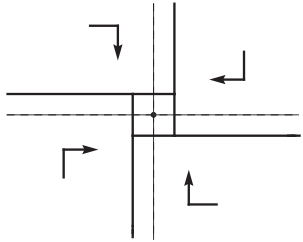}
\caption{Division of the grid in the proof of Proposition \ref{prop:optcentral}.}
\label{fig:squarediv}
\end{center}
\end{figure}

Let us now compute the running time in the $r$-central case. It is obvious
that using this algorithm all packets are sent to the 4 axis outgoing from
the central vertex. The congestion of the edge in the axis containing the
central vertex along each line is $1+2+3+\ldots+r={r+1 \choose 2}$. Since
at each step one packet reaches its destination along each line, we
conclude that ${r+1 \choose 2}$ is the total running time of the
algorithm.
\end{proof}

%V-O modified
\paragraph{Triangular grid}
The same idea  applies to the triangular grid. In this
case, the number of nodes at distance at most $r$ is $6{r+1 \choose 2}$.
The cut is made of $6$ edges. Dividing the plane onto 6 subregions gives
again an optimal algorithm performing in ${r+1 \choose 2}$ steps.
\paragraph{Hexagonal grid}

The same idea gives an optimal routing in the $r$-central case. In this
case the degree of each vertex is three, and then it is easy to check (maybe a
drawing using \textsc{Fig.} \ref{fig:axis_hex} can help) that there are
$3{r+1 \choose 2}$ nodes at distance at most $d$ that may want to send a
message to the central vertex, and the cut has size 3. As expected, the
running time is again ${r+1 \choose 2}$.

%\newpage

\section{$(\ell,k)$-Routing}
\label{sec:lk}

%V-O slightly modified
Recall that in the general $(\ell,k)$-routing problem each node sends
at most $\ell$ packets and receives at most $k$ packets. We propose a
distributed approximation algorithm using the ideas of the previous algorithms for the permutation routing problem. We also provide
lower bounds for the running time of any algorithm using shortest path
routing, that allow us to prove that our algorithm is tight when $\ell=k$,
on any grid.

%V-O removed
%\begin{observation}
%We also propose in  \ref{sec:weighted} an approach to find a
%solution of the $(\ell,k)$-routing problem, on any grid, using the problem
%of \textsc{Weighted Edge Coloring} in a bipartite graph. However, the
%algorithm obtained using this approach is centralized. \end{observation}
%\subsection{Distributed Algorithm Based on Permutation Routing}
%\label{sec:algo_lk}

%V-O modified

We start by describing the results for full-duplex triangular grids. (The results can be adapted to square grids.)
%, but we focus on the other grids since there were few results in the literature.
We also show how to adapt the results to hexagonal grids and to the half-duplex version. In this section
we denote $c:=\PartIntSup{\frac{\max\{\ell,k\}}{\min\{\ell,k\}}}=\PartIntSup{\max\{\frac{\ell}{k},\frac{k}{\ell}\}}$. Note that $c\geq 1$. Lemma \ref{lem:LB1} and Lemma \ref{lem:LB2} provide two lower bounds for the running time of any algorithm using shortest paths.\\

\begin{lemma}[First lower bound]
The worst-case running time of any algorithm for $(\ell,k)$-routing on
full-duplex triangular grids using shortest path routing satisfies
\label{lem:LB1}
$$\text{Running time}\geq  \min\{\ell,k\}\cdot\ell_{\max}$$
\end{lemma}
\begin{proof}
Consider a set of $\ell_{\max}$ nodes placed along a line, placed
consecutively at one side of a distinguished edge $e$. Each node wants to
send $\min\{\ell,k\}$ messages to the nodes placed at the other side of
$e$ along the line, at distance $\ell_{\max}$ from it. Then the congestion
of $e$ is $\min\{\ell,k\}\cdot\ell_{\max}$, giving the bound.
\end{proof}

%\begin{observation}
%The upper bound can be easily tighten till
%$$\max\{\ell,k\}\cdot(\ell_{\max}-1)+\min\{\ell,k\},$$
%and thus the improved approximation ratio would be
%$$\frac{\max\{\ell,k\}\cdot(\ell_{\max}-1)+\min\{\ell,k\}}{\min\{\ell,k\}\cdot\ell_{\max}}$$
%\end{observation}

\begin{definition}
Given a vertex $v$, we call the rectangle of side $(a,b)$ starting at $v$
the set $R^v_{a,b}=\{v+ \alpha\text{\emph{\textbf{i}}} + \beta
\text{\emph{\textbf{j}}}, 0\leq \alpha < a, 0\leq \beta < b\}$. We call
such a rectangle a square if $a=b$. Notice that in the triangular grid the
node set is generated by
$\{\emph{\textbf{i}},\emph{\textbf{j}},\emph{\textbf{k}}\}$, where
$\emph{\textbf{k}}= -\emph{\textbf{i}} - \emph{\textbf{j}}$, as we have
explained in Section \ref{sec:tri}.
\end{definition}

Using standard graph terminology, given a graph $G=(V,E)$ and a
subset $S \subseteq V$, the set $\Gamma(S)$ denotes the (open)
neighborhood in $G$ of the vertices in $S$. The following theorem
can be found, for example, in \cite{art95}.
\begin{theorem}[Corollary of Hall's theorem \cite{art95}]
\label{teo:cor_hall} Let $G=(V,E)$ be a bipartite graph, with $V = X \cup
Y$. If for all subsets $A$ of $X$, $|\Gamma(A)| \geq c|A|$, then for each
$x\in X$, there exists $S_x \subset Y$ such that $|S_x|=c$, and $\forall
x, x' \in X$, $S_x \cap S_{x'} = \emptyset$ and $S_x \subset \Gamma(x)$.
\end{theorem}
We use this theorem to prove the following lower bound.
% If $c\leq \ell_{\max}$, $\sum_{i=1}^{c-1} i*k+(\ell_{\max}-c+1)*\ell$ and if $c > \ell_{\max}$, $\sum_{i=1}^{\ell_max} i*k$.
\begin{lemma}[Second lower bound]
\label{lem:LB2} The worst-case running time of any algorithm for
$(\ell,k)$-routing on full-duplex triangular grids using shortest path
routing satisfies
$$
\text{Running time}\geq \PartIntSup{\frac{\max\{\ell,k\}}{4} \cdot \left
\lfloor \frac{\ell_{\max}+1}{\sqrt{c+1}} \right \rfloor }\ ,
$$
where $c=\PartIntSup{\frac{\max\{\ell,k\}}{\min\{\ell,k\}}}$.
\end{lemma}

\begin{proof}
Suppose without loss of generality that $\ell \geq k$, otherwise change the role of
$\ell$ and $k$. Let $v$ be a vertex, and consider the square
$R^v_{d,d}$, with $d:= \left \lfloor \frac{\ell_{\max}+1}{\sqrt{c+1}}
\right \rfloor$. We claim that all nodes inside this square can send
$\ell$ messages such that all destination nodes are in the destination set
$D=R^v_{d+\ell_{\max},d+\ell_{\max}} \setminus R^v_{d,d}$. Let $S$ be the
subgrid generated by positive linear combinations of the vectors
\textbf{i} and \textbf{j}. More precisely, $S:=\{v+
\alpha\text{\emph{\textbf{i}}} + \beta \text{\emph{\textbf{j}}},
\alpha\geq 0, \beta \geq 0\}$. \textsc{Fig.} \ref{fig:2arcs_axis}b gives a
graphical illustration.

To prove this, we consider a bipartite graph $H$ on vertex set $R^v_{d,d}
\cup D$, with an edge between a vertex of $R^v_{d,d}$ and a vertex of $D$
if they are at distance at most $\ell_{\max}$ in $S$. To apply Theorem
\ref{teo:cor_hall}, we have to show that any subset of vertices $A \subset
R^v_{d,d}$ has at least $c|A|$ neighbors in $H$. Theorem
\ref{teo:cor_hall} will then ensure the existence of a feasible
repartition of the messages from vertices of $R^v_{d,d}$ to those of $D$
such that all shortest paths have length at most $\ell_{\max}$.

Given $A\subset R^v_{d,d}$, let us call $D_A:=\{u \in D\
:\text{dist}_{S}(A,u) \leq \ell_{\max}\}$, where $\text{dist}_{S}(A,u)$
means the minimum distance in $S$ from any vertex of $A$ to the vertex
$u$. For any $A\subset R^v_{d,d}$, we need to show that

\begin{equation}
\label{eq:expansion} |D_A|\geq c|A|\
\end{equation}
\vspace{0.00cm}
% where $D_A=\Gamma_{\ell_{\max}}(A)\cap D$ and
%$\Gamma_{\ell_{\max}}(A)=\{u:dist_{\text{right down}}(A,u) \leq
%\ell_{\max}\}$.

Without loss of generality we suppose that $A$ is maximal, in the sense
that there is no set $A'$ strictly containing $A$ with $D_A=D_{A'}$.
Instead of considering all possible sets $A$, we will show below that we
can restrict ourselves to rectangles. Hence given a set $A$, we denote by
$R_A$ the smallest rectangle containing the subset of vertices $A$. We first
claim that

\begin{equation}
\label{eq:claim} |D_{R_A} \setminus D_A| \leq |R_A \setminus A|
\end{equation}
\vspace{0.00cm}

Indeed, this equality can be shown by induction on $|R_A\setminus A|$. For
$|R_A\setminus A|=0$ the equality is trivial. Suppose that it is true for
$|R_A\setminus A|$. The induction step
consists in showing that there is an element $x$ in $R_A\setminus A$ such that $|D_{R_A}
\setminus D_{A\cup \{x\}}|-|D_{R_A}
\setminus D_A|\leq 1$ (note that $D_{R_{A \cup \{x\}}}=D_{R_A}$):\\
% Since $R_A$ is the smallest rectangle containing $A$, and
% $A$ is maximal, there are only the following three possibilities for such
% a vertex $x$:
% We consider the three cases in this order to make the computation easier. Otherwise, some $x$ could add more than one neighbor:\\

\begin{itemize}
\item[$\bullet$] If there exists $x$ such that $x+\emph{\textbf{j}}$ and $x-\emph{\textbf{i}}$ are in $A$ and $x-\emph{\textbf{j}}$ is not in $A$, then we select this $x$.  From $x$ the only new vertex we may add to $D_A$ is $x+\ell_{\max}\emph{\textbf{i}}$.
\item[$\bullet$] Otherwise, if there exists $x$ such that $x-\emph{\textbf{j}}$ and $x-\emph{\textbf{i}}$ are in $A$ and $x+\emph{\textbf{i}}$ is not in $A$, then we select this $x$. In this case the only new vertex we may add to $D_A$ is $x-\ell_{\max}\emph{\textbf{k}}$.
\item[$\bullet$] If none of the previous cases holds, since $R_A$ is the smallest rectangle containing $A$, and
 $A$ is maximal, then necessarily there exists an $x$ such that $x+\emph{\textbf{i}}$ and $x-\emph{\textbf{j}}$ are in $A$ and $x-\emph{\textbf{i}}$ is not in $A$. We select this $x$, and the only new vertex we may add to $D_A$ is $x+\ell_{\max}\emph{\textbf{j}}$.\\
\end{itemize}

\noindent Thus, in all cases there exists an $x$ adding at most one neighbor to $D_{R_A}
\setminus D_A$, which finishes the induction step and proves Equation (\ref{eq:claim}). To finish the proof of the fact that we can restrict ourselves to rectangles, we
show that, for any subset $A$, if Inequality (\ref{eq:expansion}) holds
for $R_A$, then it also holds for $A$. Indeed, Inequality (\ref{eq:expansion}) applied to $R_A$
gives:

$$c |R_A| \leq |D_{R_A}|\ ,\ \text{ which is equivalent to}$$

\begin{equation}
\label{eq:equiv} c (|A| + |R_A \setminus A|)\leq |D_A| +|D_{R_A} \setminus
D_A|
\end{equation}
\vspace{0.00cm}

Using Inequality (\ref{eq:claim}) and the fact that $c \geq 1$, Inequality
(\ref{eq:equiv}) clearly implies that Inequality (\ref{eq:expansion}) holds.

Henceforth we assume that $A$ is a rectangle. The last simplification consists in proving that we can restrict ourselves
to rectangles containing $v$. In other words, it will be sufficient to prove Inequality
(\ref{eq:expansion}) for all rectangles $R^v_{a,b}$. Given a rectangle $R$ not positioned at $v$, the rectangle $R'$ of the same size positioned at $v$ has less neighbors, hence if Inequality (\ref{eq:expansion})
holds for $R'$, it also holds for $R$.

Finally let us prove that Inequality (\ref{eq:expansion}) holds for all
rectangles $R^v_{a,b}$, with $1\leq
a,b < d$. We have  $|R^v_{a,b}|=ab$ and
$|D_{R^v_{a,b}}|=(a+\ell_{\max})(b+\ell_{\max})-d^2$.
%We shall prove that $c
%|R^v_{a,b}| \leq |D_{R^v_{a,b}}|$.
By the choice of $d$, starting from the Inequality $d^2 \leq
\frac{(\ell_{\max}+1)^2}{c+1}$ and using that $1\leq a,b$, one obtains that
$d^2c \leq (\ell_{\max}+a)(\ell_{\max}+b)-d^2$ for any $1\leq a,b < d$. This
implies, using $a,b < d$, that $cab \leq (a+\ell_{\max})(b+\ell_{\max})-d^2$
for any $1\leq a,b < d$, hence Inequality (\ref{eq:expansion}) (i.e. $c |R^v_{a,b}| \leq
|D_{R^v_{a,b}}|$) holds.

So by Theorem \ref{teo:cor_hall}, each one of the $d^2$ nodes in
$R^v_{d,d}$ can send $\ell$ messages to the nodes of $D$. Since the
number of edges going from $R^v_{d,d}$ to $D$ is $4d-1$, we apply the
bisection bound discussed in Section \ref{sec:bounds} to conclude that
there is an edge of the border of the square $R^v_{d,d}$ with
congestion at least $\PartIntSup{\frac{\ell \cdot d^2}{4 d-1}}
>\PartIntSup{ \frac{\ell \cdot d}{4}}$. This finishes the proof of the lemma. \end{proof}

We observe that this second lower bound is strictly better than the first one if
and only if

$$
\frac{c}{\sqrt{c+1}}> \frac{4\ell_{\max}}{\ell_{\max}+1}
$$

If both $c$ and $\ell_{\max}$ are big, the condition becomes approximately:

$$\frac{\max\{\ell,k\}}{\min\{\ell,k\}}>16$$

That is, the second lower bound is better when the difference between
$\ell$ and $k$ is big. This is the case of broadcast or gathering, where
messages are originated (or destined) from (or to) a small set of nodes of
the network.\\

%\item[$\bullet$] \textbf{:}
The two lower bounds can be combined to give:
\begin{lemma}[Combined lower bound] The worst-case running time of any algorithm for $(\ell,k)$-routing on full-duplex
triangular grids using shortest path routing satisfies
\begin{eqnarray*}
\text{Running time}&\geq&  \max \left( \ell_{\max} \cdot \min\{\ell,k\},\
 \max\{\ell,k\}\cdot \left \lfloor\frac{\ell_{\max}+1}{4\sqrt{c+1}}
\right \rfloor\right)\\
&\approx& \ell_{\max} \cdot \max \left( \min\{\ell,k\},\
 \frac{\max\{\ell,k\}}{4\sqrt{c+1}}
\right)
\end{eqnarray*}

\end{lemma}
%\max \left(\min\{\ell,k\}\cdot\ell_{\max}\ ,\ \ell_{\max}\cdot \ell \cdot \frac{\sqrt{ \lceil l/k \rceil}}{4}\right) %$$
%$$
Now we provide an algorithm from which we derive an upper bound.
\begin{proposition}[Upper bound (algorithm)]
The algorithm for $(\ell,k)$-routing on full-duplex triangular grids is
the following: route all packets as in the permutation routing case. That
is, at each node send packets first in their negative component, breaking
ties arbitrarily (there can be $\ell$ packets in conflict in a negative
component). If there are no packets with negative components, send any of
the (at most $k$) packets with maximum remaining distance.
$$
\text{Running time}\leq \left\{\begin{array}{cl}
\min\{\ell,k\}\cdot\frac{c(c-1)}{2}+\max\{\ell,k\} \cdot (\ell_{\max}-c+1)
&\mbox{, if }c\leq \ell_{\max}
\\\min\{\ell,k\} \cdot \frac{\ell_{\max}(\ell_{\max}+1)}{2}&\mbox{, if }c > \ell_{\max}
\end{array}\right.
$$
\end{proposition}
\begin{proof}
Suppose again without loss of generality that $\ell \geq k$. We proceed by decreasing induction on $ \ell_{\max}$. We prove that after $\min\{\ell,
\ell_{\max}k\}$ steps, each packet will be at distance at most $\ell_{\max} -1$ of its destination. This yields

\begin{equation*}
\text{Running time}(\ell_{\max}) \leq   \min\{\ell, \ell_{\max}k\} + \text{Running time}(\ell_{\max}-1)
\end{equation*}
\begin{equation*}
\leq \min\{\ell, \ell_{\max}k\} + \left \{\begin{array}{cl}
\min\{\ell,k\}\cdot\frac{c(c-1)}{2}+\max\{\ell,k\} \cdot (\ell_{\max}-c)
&\mbox{, if }c\leq \ell_{\max}-1
\\\min\{\ell,k\} \cdot \frac{\ell_{\max}(\ell_{\max}-1)}{2}&\mbox{, if }c > \ell_{\max}-1
\end{array}\right.
\end{equation*}
\begin{equation*}
\leq \left \{\begin{array}{cl}
\min\{\ell,k\}\cdot\frac{c(c-1)}{2}+\max\{\ell,k\} \cdot (\ell_{\max}-c+1)
&\mbox{, if }c\leq \ell_{\max}
\\\min\{\ell,k\} \cdot \frac{\ell_{\max}(\ell_{\max}+1)}{2}&\mbox{, if }c > \ell_{\max}
\end{array}\right.
\end{equation*}

 Let us consider the messages at distance $\ell_{\max}$ to their destinations. They are of two types, the one moving according to their negative component and the one moving according to their positive component.

If $c\leq \ell_{\max}$ the first ones move after at most $\ell$ time steps. If $ c < \ell_{\max}$ they move more quickly,
indeed they move at least once every $\ell_{\max}k$ steps ($\ell_{\max}k \leq c \cdot k=\ell$). This is due to the fact that when $c < \ell_{\max}$ at a given vertex, at most $\ell_{\max}k$ messages may have to move according to their negative component toward a node at distance $\ell_{\max}$.

About the messages which move according to their positive component, since a node is the destination of at most $k$ messages, they may wait at most $k$ steps.

Consequently, $\ell_{\max}$ decreases by at least one every $\min\{\ell,
\ell_{\max}k\}$ steps, which gives the result.
\end{proof}

This gives an algorithm which is fully distributed. Dividing the running
time of this algorithm by the combined lower bound we obtain the
following ratio:

%This gives an algorithm which is fully distributed and has runing time within at most a factor
$$
\left\{\begin{array}{cl} \frac{\min\{\ell,k\}\cdot{c \choose
2}+\max\{\ell,k\} \cdot (\ell_{\max}-c+1)}{\ell_{\max} \cdot \max
\left(\min\{\ell,k\}\ ,\ \frac{\max\{\ell,k\}}{4 \sqrt{c+1}}\right)}
&\mbox{, if }c\leq \ell_{\max}
\\\frac{\min\{\ell,k\} \cdot (\ell_{\max}+1)}{2 \cdot \max \left(\min\{\ell,k\}\ ,\ \frac{\max\{\ell,k\}}{4 \sqrt{c+1}}\right)}&\mbox{, if }c > \ell_{\max}
\end{array}\right.
$$
%of any optimal $(\ell, k)$-routing algorithm.

%\textbf{Upper bound}: route all packets as in the permutation
%routing case. Running time:
%$$\max\{\ell,k\}\cdot\ell_{\max}$$
%\begin{proof}
%For each packet, let $l_1$ be the number of steps routed along the
%first direction, and let let $l_2=l_p-l_1$ be the number of steps
%routed along the second direction. It is easy to see that the
%running time is upperbounded by $\ell \cdot l_1 + k \cdot
%(\ell_{\max}-l_1)\leq \max\{\ell,k\}\ell_{\max}$.
%\end{proof}

%\textbf{Can the analysis of the approximation ratio be
%improved?}\\

%Remark that this algorithm is fully distributed.

%\subsection{Triangular grid}
%\subsection{Hexagonal grid
We observe that in all cases the running time of the algorithm is at most
$\max\{\ell,k\}\cdot\ell_{\max}$. In particular, when $\ell=k$ (that is,
$c=1$) the running time is exactly
$\max\{\ell,k\}\cdot\ell_{\max}=\min\{\ell,k\}\cdot\ell_{\max}$, and
therefore it is tight (see lower bound of Lemma \ref{lem:LB1}).
\begin{corollary}
There exists a tight algorithm for $(k,k)$-routing in full-duplex
triangular grids. \end{corollary}

The previous algorithms can be generalized for half-duplex triangular
grids as well as for full and half-duplex hexagonal grids. The
generalization to half-duplex grids is obtained by just adding a factor 2
in both the lower bound and the running time of the algorithm, as we did
for the permutation routing algorithm. Thus, let us just focus on the case
of full-duplex hexagonal grids, for which we have the following theorems:

\begin{theorem}
There exists an algorithm for $(\ell,k)$-routing in full-duplex hexagonal
grids whose running time is at most:
$$
\text{\emph{Running time}}\leq \left\{\begin{array}{cl}
2 \min\{\ell,k\}\cdot\frac{c(c-1)}{2}+2\max\{\ell,k\} \cdot (\ell_{\max}-c+1)    & \mbox{, if } c \leq \ell_{\max}\\
2 \min\{\ell,k\} \cdot \frac{\ell_{\max}(\ell_{\max}+1)}{2}           & \mbox{, if } c >    \ell_{\max}
\end{array}\right.
$$
\end{theorem}

\begin{lemma}[First lower bound]
No algorithm based on shortest path routing can route all messages using
less than $2 \min\{\ell,k\}\cdot\ell_{\max} - \min\{\ell,k\}$ steps in the
worst case.
\end{lemma}

\begin{definition}
Given a vertex $v$, we call the \emph{rectangle} of the hexagonal grid of
side $(a,b)$ starting at $v$ to the subset of the hexagonal grid
${R_{hex}^v}_{a,b}=\{v+ \alpha \text{\emph{\textbf{i}}} + \beta
\text{\emph{\textbf{j}}} + \gamma \text{\emph{\textbf{k}}}, 0\leq \alpha <
a, -\gamma < \beta < b, 0\leq \gamma < b \} \cap H$ where $H$ is the
vertex set of the hexagonal grid. We call such a rectangle a \emph{square}
if $a=b$.
\end{definition}

%\begin{theorem}
%We have a second lower bound on the running time of an algorithm using
%shortest path routing:
%$$\text{Running time}\geq 2d+\frac{d-2}{2d+1}$$
%where $d:%\frac{\sqrt{73c+64\max(\ell,k)^2+121+144\max(\ell,k)}}{8\sqrt{c+1}} -
%\frac{3}{8}$.
%\end{theorem}
%\begin{proof}
%The proof consists in showing that for $d%\frac{\sqrt{73c+64\max(\ell,k)^2+121+144\max(\ell,k)}}{8\sqrt{c+1}} -
%\frac{3}{8}$, the vertices of ${R_{hex}^v}_{d,d}$ can simultaneously send
%$\max(\ell,k)$ messages to some vertices of
%${R_{hex}^v}_{d+\ell_{\max},d+\ell_{\max}} \setminus {R_{hex}^v}_{d,d}$.
%This is done as for the triangular grid, using again Theorem
%\ref{teo:cor_hall}. Since the number of vertices inside
%${R_{hex}^v}_{d,d}$ is $4d^2+d-2$, and the number of edges outgoing from
%${R_{hex}^v}_{d,d}$ is $2d+1$, we have a congestion of
%$\frac{4d^2+d-2}{2d+1}=2d+\frac{d-2}{2d+1}$.
%\end{proof}
The following lemma gives a second lower bound on the running time of any
algorithm using shortest path routing on full-duplex hexagonal grids.
\begin{lemma}[Second lower bound]
The worst-case running time of any algorithm using shortest path routing
on full-duplex hexagonal grids satisfies:
$$\text{Running time}\geq  \PartIntSup{\max\{\ell,k\}(2d+\frac{d-2}{2d+1})}\ \ ,$$
where $d=\PartIntInf{
\frac{\sqrt{73c+64\ell_{\max}^2+121+144\ell_{\max}}}{8\sqrt{c+1}} -
\frac{3}{8}}$ and $c=\PartIntSup{\frac{\max\{\ell,k\}}{\min\{\ell,k\}}}$.
\end{lemma}

Notice that when $\frac{\ell_{\max}}{c}$ is big, this value tends to $2
 \max\{\ell,k\} \frac{\ell_{\max}} {\sqrt{c+1}}$, obtaining a performance
approximately twice better than in triangular grids.

\begin{proof}
The proof consists in showing that %for $d:=\PartIntInf{\frac{\sqrt{73c+64\ell_{\max}^2+121+144\ell_{\max}}}{8\sqrt{c+1}}- \frac{3}{8}}$,
the vertices of ${R_{hex}^v}_{d,d}$ can simultaneously
send $\max(\ell,k)$ messages to some vertices of
${R_{hex}^v}_{d+\ell_{\max},d+\ell_{\max}} \setminus {R_{hex}^v}_{d,d}$.
This is done as for the triangular grid, using again Theorem
\ref{teo:cor_hall}. We do not give all the details, since the idea behind
is the same as the proof of Lemma \ref{lem:LB2}.

Since the number of vertices inside ${R_{hex}^v}_{d,d}$ is $4d^2+d-2$, and
the number of edges outgoing from ${R_{hex}^v}_{d,d}$ is $2d+1$, the
congestion on these edges is $\max\{\ell,k\}\frac{4d^2+d-2}{2d+1}= \max\{\ell,k\}(2d+\frac{d-2}{2d+1})$.
\end{proof}

% We mention that it is possible to obtain a more complicated formula for
% the approximation ratio.

%In this case the formula of the approximation ratio becomes too
%complicated and we have decided not to include it at this point.
%\newpage

\section{Conclusions and Further Research}
 \label{sec:concl}

In this article we have studied the permutation routing, the $r$-central
routing and the general $(\ell,k)$-routing problems on plane grids, that
is square grids, triangular grids and hexagonal grids. We have assumed the
\emph{store-and-forward} $\Delta$-port model, and considered both full
and half-duplex networks. The main new results of this article are the
following:\\

\begin{itemize}

\item[1.] Tight (also including the constant factor) permutation routing
algorithms on
full-duplex hexagonal grids, and half duplex triangular and hexagonal
grids.
\item[2.] Tight (also including the constant factor) $r$-central routing
algorithms on triangular and hexagonal grids.

 \item[3.] Tight (also including the constant factor) $(k,k)$-routing
algorithms on
square, triangular and hexagonal grids.

\item[4.] Good approximation algorithms for $(\ell,k)$-routing in
square, triangular and hexagonal grids, together with new lower bounds on
the running time of any algorithm using shortest path routing.\\
\end{itemize}
\noindent All these algorithms are completely distributed, i.e., can be
implemented independently at each node. Finally, in~\ref{sec:weighted}, we have formulated
the $(\ell,k)$-routing problem as a \textsc{Weighted Edge Coloring}
problem on bipartite graphs.

There still remain several interesting open problems concerning
$(\ell,k)$-routing on plane grids. Of course, the most challenging problem
seems to find a tight $(\ell,k)$-routing algorithm for any plane grid, for
$\ell \neq k$. Another interesting avenue for further research is to take
into account the queue size. That is, to devise $(\ell,k)$-routing
algorithms with bounded queue size, or that optimize both the running time
and the queue size, under a certain trade-off.

\vspace{0.1cm}

\section*{Acknowledgments}
 We want to thank Fr\'ed\'eric Giroire,
Cl\'audia Linhares-Sales and Bruce Reed for insightful discussions. A special thank to Omid Amini, for both insightful discussions and his help to improve previous versions of this paper.

This work has been partially supported by European project IST FET AEOLUS, PACA region of France, Ministerio de Educaci\'on y Ciencia of Spain,  European Regional Development Fund under project TEC2005-03575, Catalan Research Council under project 2005SGR00256, Slovenian Research Agency ARRS, and  COST action 293 GRAAL, and has been done in the context of the  {\sc crc Corso} with France Telecom.

\begin{appendix}
\section{Defining the embeddings}

\label{sec:embed}

The results already known for the square grid can be used for a triangular
(resp. a hexagonal) grid if we have an adapted function mapping the square
grid into the triangular (resp. hexagonal) grid. Here we propose both
functions, namely \emph{square2triangle} and \emph{square2hexagon}.

The function \emph{square2triangle} is illustrated in \textsc{Fig.}
\ref{fig:s2t}. We perform the same routing as in the grid, i.e. we just
ignore the extra diagonal.

\vspace{-.2cm}
\begin{figure}[h!tb]
\begin{center}
\includegraphics[angle=-90,width=5.8cm]{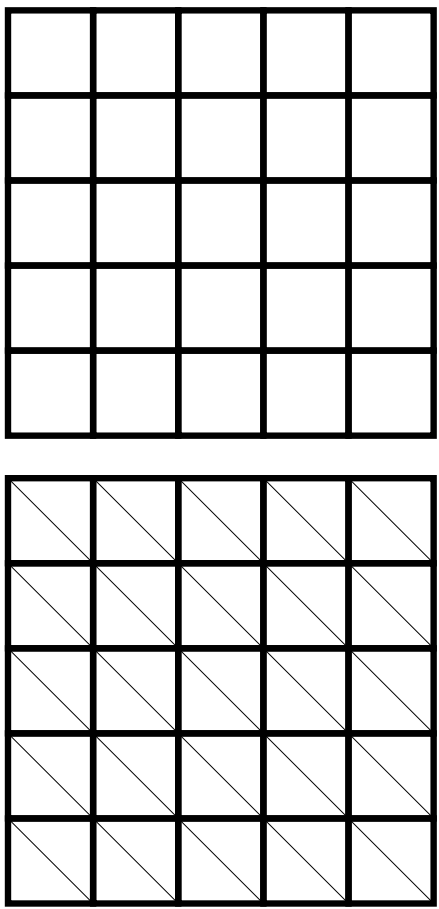}
\caption{Square grid mapped into the triangular grid. \label{fig:s2t}}\vspace{-.3cm}
\end{center}
\end{figure}

In a square grid the distance between two vertices is at most twice the
distance in a triangular grid. Similarly the congestion is at most doubled
going from the triangular grid to the square grid. Nevertheless the
maximal distance is unchanged. Indeed, the NW and SE nodes of
\textsc{Fig.} \ref{fig:s2t} are at the same distance in both grids.
Consequently, an algorithm which routes a permutation in $2n-2$ steps is
still optimal in the worst case. If instead of
considering a square grid with one extra diagonal, we look at a triangle
grid as on \textsc{Fig.} \ref{fig:triangle}, or in shape of a triangle,
Using the routing of the square grid in this triangle grid yields a
routing within twice the optimal (i.e. minimum time in the worst case).

\begin{figure}[h!tb]
\begin{center}
\includegraphics[width=4cm]{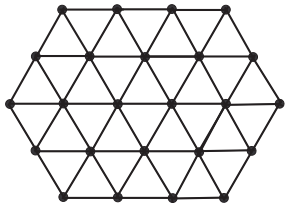}
\caption{Triangular grid.} \label{fig:triangle} \vspace{-.3cm}
\end{center}
\end{figure}

The function \emph{square2hexagon} is a little more complicated. Squares
are mapped in two different ways on the hexagonal grid, as shown in
\textsc{Fig.} \ref{fig:s2h}. Some are mapped on the left side of a hexagon
and some on the right side. Call them respectively \emph{white} and
\emph{black} squares. White and black squares alternate on the grid like
white and black on a chess board. The missing edge of a white square is
mapped to the path of length 3 that goes on the right of the hexagon. The
missing edge of a black square is mapped on the path of length 3 that goes
on the right of the hexagon. In this way each edge of the square grid is
uniquely mapped and each edge of the hexagonal grid is the image of
exactly two edges of the square grid.

\begin{figure}[h!tb]
\begin{center}
\includegraphics[angle=-90,width=8.0cm]{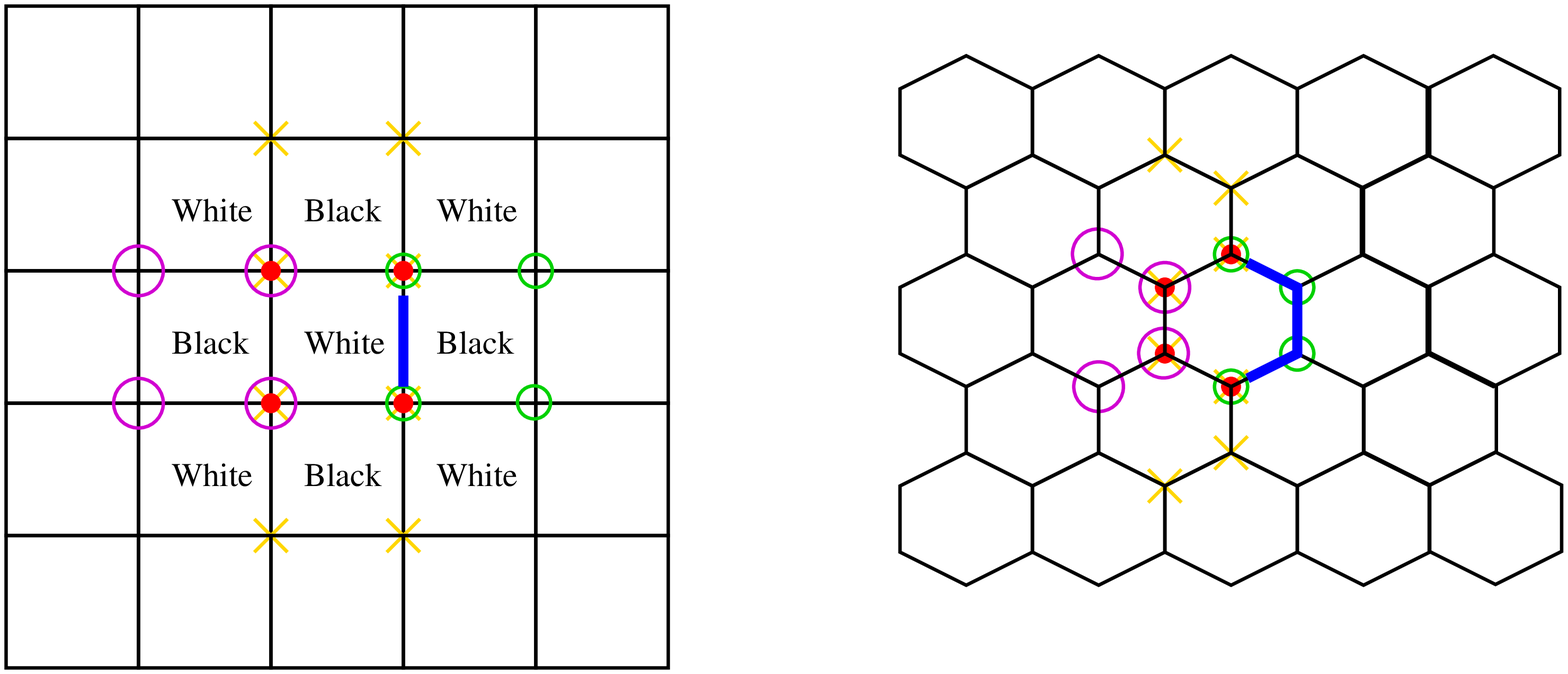}
\caption{Square grid mapped into the hexagonal grid.} \label{fig:s2h}\vspace{-.3cm}
\end{center}
\end{figure}
%\begin{figure}[h!tb]
%\begin{center}
%\includegraphics[angle=-90,width=12cm]{Fig/s2hbis}
%\caption{Square grid mapped into the hexagonal grid} \label{fig:s2h}
%\end{center}
%\end{figure}
The distance between 2 vertices in the hexagonal grid is twice the
distance in the square grid plus one. Also when we adapt a routing from
the square grid to the hexagonal grid using the function
\emph{square2hexagon}, the congestion may double since each edge of the
hexagonal grid is the image of two edges of the square grid. Consequently,
a routing obtained using the function \emph{square2hexagon} will be within
a constant multiplicative factor of the optimal.

\section{An Approach for $(\ell,k)$-routing Using Weighted Coloring} \label{sec:weighted}

%\subsubsection{$(\ell,k)$-routing as a Weighted Coloring Problem}
%\label{sec:model}

In any physical topology, we can represent a given instance of the problem
in the following way. Given a network on $n$ nodes, we build a bipartite
graph $H$ with a copy of each node at both sides of the bipartition. We
add an edge between $u$ and $v$ whenever $u$ wants to send a message to
$v$, and assign to each edge $uv$ a weight $w(uv)$ equals to the length of
a shortest path from $u$ to $v$ on the original grid. In this way we
obtain an edge-weighted bipartite graph $H$ on $2n$ nodes. Note that the
maximum degree of $H$ satisfies $\Delta \leq \max\{\ell,k\}$. An example
 for $\ell=2$ and $k=3$ is depicted in \textsc{Fig.}
\ref{fig:examplebipartite}.

\begin{figure}[h!tb]
\begin{center}
\includegraphics[width=4.5cm]{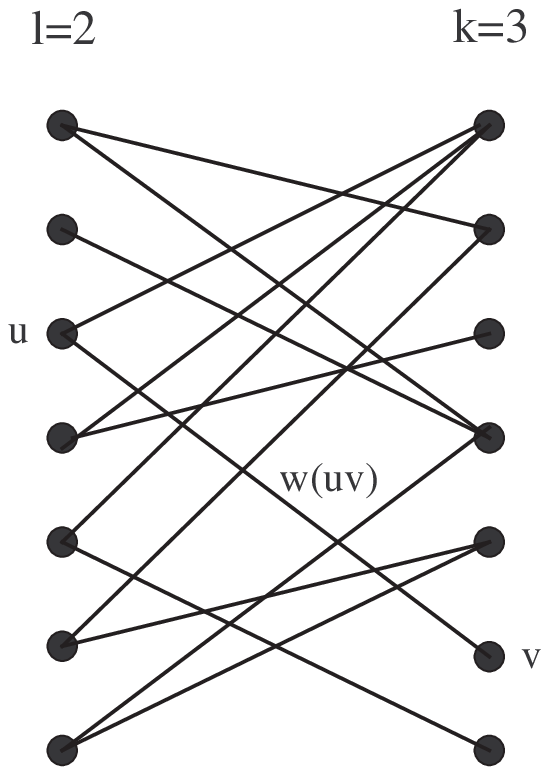}
\caption{Bipartite graph modeling a $(2,3)$-routing instance.} \label{fig:examplebipartite}
\end{center}
\end{figure}

The key idea behind this construction is that each matching in $H$
corresponds to an instance of a permutation routing problem. Hence, it can
be solved optimally, as we have proved for all types of grids in Section
\ref{sec:perm_rout}. For each matching $M_i$, we define its cost as
$c(M_i):=\max\{w(e)|e\in M_i)\}$. We assign this cost because on all grids
the running time of the permutation routing algorithms we have described
are proportional to the length of the longest shortest path (with equality
on full-duplex triangular grids).

>From the classical Hall's theorem we know that the edges of a bipartite
graph can be partitioned into $\Delta$ disjoint matchings (that is, a
coloring of the edges), $\Delta$ being the maximum degree of the graph. In
our case we have  $\Delta=\max\{\ell,k\}$. Thus, the problem consists in
partitioning the edges of $H$ into $\Delta$ matchings
$M_1,\ldots,M_{\Delta}$, in such a way that $\sum_{i=1}^{\Delta}c(M_i)$ is
minimized. That is, our problem, namely
\textsc{Weighted Bipartite Edge Coloring}, can be stated in the following way:\\

\textsc{Weighted Bipartite Edge Coloring}\\
\textbf{Input:} An edge-weighted bipartite graph $H$.\newline
\textbf{Output:} A partition of the edges of $H$ into matchings
$M_1,\ldots,M_\Delta$, with $c(M_i):=\max\{w(e)|e\in M_i)\}$.\newline
\textbf{Objective:} $\min \sum_{i=1}^{\Delta}c(M_i)$.\\

Therefore, $\min \sum_{i=1}^{\Delta}c(M_i)$ is the running time for
routing an $(\ell,k)$-routing instance using this algorithm.

\iffalse Unfortunately, it turns out that finding such a weighted
coloring is \textsc{NP}-complete (cf. \cite{kesselmanNPhard,
TowlesNPhard} for an analysis arising from the scheduling of
optical switches, and
\cite{weighted_coloring_expensive,weighted_coloring_approx_bipartite}
for a graph theoretical approach). Hence, it makes sense to look
for approximation algorithms. We provide first a brief
state-of-the art about \textsc{Weighted Coloring}.

%and this is the topic of the next section, in which we give a centralized
%algorithm based on the problem of \textsc{Weighted Edge Coloring}.

%\subsubsection{Algorithm based on edge coloring}
%\label{sec:edgecolor}

The problem of \textsc{Weighted Coloring} was first introduced in
\cite{weighted_coloring_intro}. Results on complexity and
approximation algorithms are given in
\cite{weighted_coloring_approx} for general graphs, and in
\cite{weighted_coloring_expensive} for bipartite, planar and split
graphs. Finally, in \cite{Claudia_Bruce} the existing results on
graphs with bounded tree-width are improved.

The problem of \textsc{Weighted Edge Coloring} can be easily reduced to
the problem of \textsc{Weighted Coloring} (i.e. \textsc{Weighted Node
Coloring}) in the following way: given an edge-weighted graph $G$,
consider its line graph $L(G)$, where each node in $L(G)$ has the weight
of the corresponding edge in $G$. Then, there is a one-to-one
correspondence between weighted edge colorings of $G$ and weighted (node)
colorings of $L(G)$. \fi

Unfortunately, in \cite{weighted_coloring_approx_bipartite}
\textsc{Weighted Edge Coloring} is proved to be strongly
\textsc{NP}-complete for bipartite graphs, which is the case we
are interested in. In fact, the problem remains strongly
NP-complete even restricted to cubic and planar bipartite graphs.
Concerning approximation results, the authors
\cite{weighted_coloring_approx_bipartite} provide an
inapproximability bound of $\frac{7}{6}-\varepsilon$, for any
$\varepsilon>0$. Furthermore, they match this bound with an
approximation algorithm within $7/6$ on graphs with maximum degree
3, improving the best known approximation ratio of $5/3$
\cite{weighted_coloring_expensive}. In \cite{kesselmanNPhard} this
innaproximability bound is proved independently on general
bipartite graphs. Thus, if $\max\{\ell,k\}\leq 3$ we can find a
solution of \textsc{Weighted Bipartite Edge Coloring} within
$\frac{7}{6}$ times the optimal solution, and this will be also a
solution for the $(\ell,k)$-routing problem.

\begin{observation}
Although of theoretical value, the main problem of this algorithm is that
finding these matchings is a \textbf{centralized} task. In addition, the
\emph{true} ratio, i.e. related to the optimum of the $(\ell,k)$-routing,
should be proved to provide an upper bound for the running time of this
algorithm.
%A possible improvement is to introduce new nodes for each
%outgoing edge and each length of the messages, obtaining  $12 \cdot
%\ell_{\max} \cdot n$ nodes in the bipartite graph, which remains
%polynomial.
\end{observation}

\end{appendix}

\bibliographystyle{abbrv}

%\fancyhead[RE,LO]{\textbf{References}}
%\addcontentsline{toc}{chapter}{References}
\bibliography{joinHSZ-Final}

\begin{thebibliography}{10}

\bibitem{AAF+97}
J.~Aspnes, Y.~Azar, A.~Fiat, S.~Plotkin, and O.~Waarts.
\newblock {On-line routing of virtual circuits with applications to load
  balancing and machine scheduling}.
\newblock {\em Journal of the ACM (JACM)}, 44(3):486--504, 1997.

\bibitem{ABD+06}
J.~Aspnes, C.~Busch, S.~Dolev, P.~Fatourou, C.~Georgiou, and A.~Shvartsman.
\newblock {Eight Open Problems in Distributed Computing}.
\newblock {\em Bulletin of the EATCS no}, 90:109--126, 2006.

\bibitem{AA94}
B.~Awerbuch and Y.~Azar.
\newblock {Local optimization of global objectives: competitive distributed
  deadlock resolution and resource allocation}.
\newblock In {\em Proceedings of 35th Annual Symposium on Foundations of
  Computer Science}, pages 240--249, 1994.

\bibitem{ACF+04}
Y.~Azar, E.~Cohen, A.~Fiat, H.~Kaplan, and H.~R{\"a}cke.
\newblock {Optimal oblivious routing in polynomial time}.
\newblock {\em Journal of Computer and System Sciences}, 69(3):383--394, 2004.

\bibitem{gathering_radio}
J.-C. Bermond, J.~Galtier, R.~Klasing, N.~Morales, and S.~P\'erennes.
\newblock {Hardness and approximation of Gathering in static radio networks}.
\newblock {\em Parallel Processing Letters}, 16(2):165--183, 2006.

\bibitem{BoHo82}
A.~Borodin and J.~Hopcroft.
\newblock {Routing, merging and sorting on parallel models of computation}.
\newblock In {\em Proceedings of the Fourteenth Annual ACM Symposium on Theory
  of Computing}, pages 338--344. ACM Press New York, NY, USA, 1982.

\bibitem{BMXb05}
C.~Busch, M.~Magdon-Ismail, and J.~Xi.
\newblock {Oblivious routing on geometric networks}.
\newblock In {\em Proceedings of the 17th annual ACM symposium on Parallelism
  in algorithms and architectures}, pages 316--324, 2005.

\bibitem{BMX05}
C.~Busch, M.~Magdon-Ismail, and J.~Xi.
\newblock {Optimal oblivious path selection on the mesh}.
\newblock In {\em Proceedings of the International Parallel and Distributed
  Processing Symposium (IPDPS05)}, 2005.

\bibitem{mesh_reconf}
J.~Cogolludo and S.~Rajasekaran.
\newblock {Permutation Routing on Reconfigurable Meshes}.
\newblock {\em Algorithmica}, 31:44--57, 2001.

\bibitem{wireless}
A.~Datta.
\newblock {A Fault-Tolerant Protocol for Energy-Efficient Permutation Routing
  in Wireless Networks}.
\newblock {\em IEEE Transactions on Computers}, 54(11):1409--1421, 2005.

\bibitem{weighted_coloring_approx_bipartite}
D.~de~Werra, M.~Demange, B.~Escoffier, J.~Monnot, and V.~T. Paschos.
\newblock {Weighted Coloring on Planar, Bipartite and Split Graphs: Complexity
  and Improved Approximation}.
\newblock {\em Lecture Notes in Computer Science}, 3341:896--907, 2004.

\bibitem{weighted_coloring_expensive}
M.~Demange, D.~de~Werra, J.~Monnot, and V.~T. Paschos.
\newblock {Weighted node coloring: when stable sets are expensive}.
\newblock {\em Lecture Notes in Computer Science}, 2573:114--125, 2002.

\bibitem{art95}
R.~Diestel.
\newblock {\em {Graph Theory}}.
\newblock Springer-Verlag, 2005.

\bibitem{Janez3}
T.~Dobravec, B.~Robi\v{c}, and J.~\v{Z}erovnik.
\newblock {Permutation routing in double-loop networks: design and empirical
  evaluation}.
\newblock {\em Journal of Systems Architecture}, 48:387--402, 2003.

\bibitem{Janez2}
T.~Dobravec, J.~\v{Z}erovnik, and B.~Robi\v{c}.
\newblock An optimal message routing algorithm for circulant networks.
\newblock {\em Journal of Systems Architecture}, 52:298--306, 2006.

\bibitem{deltaport}
P.~Fraigniaud and E.~Lazard.
\newblock Methods and problems of communication in usual networks.
\newblock {\em Discrete Applied Mathematics}, 53:79--134, 1994.

\bibitem{survey_routing98}
M.~D. Grammatikakis, M.~K. D.~F.~Hsu, and J.~Sibeyn.
\newblock {Packet Routing in Fixed-Connection Networks: a Survey}.
\newblock {\em Journal of Parallel and Distributed Processing}, 54(2):77--132,
  1998.

\bibitem{havill01online}
J.~T. Havill.
\newblock {Online Packet Routing on Linear Arrays and Rings}.
\newblock {\em Lecture Notes in Computer Science}, 2076:773--784, 2001.

\bibitem{big_foot}
F.~Hwang, T.~Lin, and R.~Jan.
\newblock {A Permutation Routing Algorithm for Double Loop Network}.
\newblock {\em Parallel Processing Letters}, 7(3):259--265, 1997.

\bibitem{low_hyper}
F.~Hwang, Y.~Yao, and B.~Dasgupta.
\newblock Some permutation routing algorithms for low-dimensional hypercubes.
\newblock {\em Theoretical Computer Science}, 270:111--124, 2002.

\bibitem{CCC}
G.~E. Jan and M.-B. Lin.
\newblock Concentration, load balancing, partial permutation routing, and
  superconcentration on cube-connected cycles parallel computers.
\newblock {\em Journal of Parallel and Distributed Compututing}, 65:1471--1482,
  2005.

\bibitem{KKT91}
C.~Kaklamanis, D.~Krizanc, and T.~Tsantilas.
\newblock {Tight bounds for oblivious routing in the hypercube}.
\newblock {\em Theory of Computing Systems}, 24(1):223--232, 1991.

\bibitem{init_prot}
D.~Karimou and J.~F. Myoupo.
\newblock {An Application of an Initialization Protocol to Permutation Routing
  in a Single-Hop Mobile Ad Hoc Networks}.
\newblock {\em The Journal of Supercomputing}, 31:215--226, 2005.

\bibitem{kesselmanNPhard}
A.~Kesselman and K.~Kogan.
\newblock {Non-Preemptive Scheduling of Optical Switches}.
\newblock In {\em Proceedings of IEEE GLOBECOM}, pages 1840--1844, 2004.

\bibitem{klavzar}
S.~Klav\v{z}ar, A.~Vesel, and P.~\v{Z}igert.
\newblock On resonance graphs of catacondensed hexagonal graphs: structure,
  coding, and hamiltonian path algorithm.
\newblock {\em MATCH Communications in Mathematical and in Computer Chemistry},
  49:99--116, 2003.

\bibitem{comp_graph}
E.~Kranakis, H.~Sing, and J.~Urrutia.
\newblock {Compas Routing in Geometric Graphs}.
\newblock In {\em 11th Canadian Conference of Computational Geometry}, pages
  51--54, 1999.

\bibitem{KNR91}
M.~Kunde, R.~Niedermeier, and P.~Rossmanith.
\newblock {Faster sorting and routing on grids with diagonals}.
\newblock {\em Lecture Notes in Computer Science}, 775:225--236, 1994.

\bibitem{Leighton_fast_95}
F.~Leighton, B.~M. Maggs, and A.~W. Richa.
\newblock {Fast Algorithms for Finding O(Congestion + Dilation) Packet Routing
  Schedules}.
\newblock In {\em 28th Annual Hawaii International Conference on System
  Sciences}, pages 555--563, 1995.

\bibitem{Leighton_C_D_94}
F.~T. Leighton, B.~M. Maggs, and S.~B. Rao.
\newblock {Packet Routing and Job-Shop Scheduling in O(congestion + dilation)
  Steps}.
\newblock {\em Combinatorica}, 14(2):167--186, 1994.

\bibitem{duplex}
T.~Leighton.
\newblock {\em {Introduction to Parallel Algorithms and Architectures:
  Arrays-Trees-Hypercubes}}.
\newblock Morgan-Kaufman, San Mateo, California, 1992.

\bibitem{LMR88}
T.~Leighton, B.~Maggs, and S.~Rao.
\newblock {Universal packet routing algorithms}.
\newblock In {\em 29th Annual Symposium on Foundations of Computer Science},
  pages 256--269, 1988.

\bibitem{constant95}
T.~Leighton, F.~Makedon, and I.~G. Tollis.
\newblock {A $2n-2$ Step Algorithm for Routing in an $n \times n$ Array with
  Constant-Size Queues}.
\newblock {\em Algorithmica}, 14:291--304, 1995.

\bibitem{optical}
W.~Liang and X.~Shen.
\newblock {Permutation Routing in All-Optical Product Networks}.
\newblock {\em IEEE Transactions on Circuits and Systems}, 49(4):533--538,
  2002.

\bibitem{Mag06}
B.~Maggs.
\newblock {A Survey of Congestion+ Dilation Results for Packet Scheduling}.
\newblock In {\em Proceedings of the 40th Annual Conference on Information
  Sciences and Systems}, volume~22, pages 1505--1510, 2006.

\bibitem{MHV+97}
B.~Maggs, F.~Mayer auf~der Heide, B.~Vocking, and M.~Westermann.
\newblock {Exploiting locality for data management in systems of limited
  bandwidth}.
\newblock In {\em Proceedings of the 38th Annual IEEE Symposium on Foundations
  of Computer Science}, pages 284--293, 1997.

\bibitem{Makedon_opt_93}
F.~Makedon and A.~Symvonis.
\newblock {Optimal algorithms for the many-to-one routing problem on
  2-dimensional meshes}.
\newblock {\em Microprocessors and Microsystems}, 17:361--367, 1993.

\bibitem{Stoj1}
F.~G. Nocetti, I.~Stojmenovi\'{c}, and J.~Zhang.
\newblock {Addressing and Routing in Hexagonal Networks with Applications for
  Tracking Mobile Users and Connection Rerouting in Cellular Networks}.
\newblock {\em IEEE Transactions on Parallel and Distributed Systems},
  13(9):963--971, 2002.

\bibitem{Universal97}
R.~Ostrovsky and Y.~Rabani.
\newblock {Universal O(congestion + dilation + $\log^{1+\varepsilon}N$) Local
  Control Packet Switching Algorithms}.
\newblock In {\em 29th Annual ACM Symposium on the Theory of Computing}, pages
  644--653, New York, 1997.

\bibitem{pietracaprina01optimal}
A.~Pietracaprina and G.~Pucci.
\newblock {Optimal Many-to-One Routing on the Mesh with Constant Queues}.
\newblock {\em Lecture Notes in Computer Science}, 2150:645--649, 2001.

\bibitem{Rac02}
H.~R\"acke.
\newblock {Minimizing congestion in general networks}.
\newblock In {\em Proceedings of the 43rd Annual IEEE Symposium on Foundations
  of Computer Science}, pages 43--52, 2002.

\bibitem{rajasekaran}
S.~Rajasekaran and R.~Overholt.
\newblock Constant queue routing on a mesh.
\newblock {\em Journal of Parallel and Distributed Computing}, 15:160--166,
  1992.

\bibitem{Janez1}
B.~Robi\v{c} and J.~\v{Z}erovnik.
\newblock {Minimum 2-terminal routing in 2-jump circulant graphs}.
\newblock {\em Computers and Artificial Intelligence}, 19(1):37--46, 2000.

\bibitem{INOC07}
I.~Sau and J.~\v{Z}erovnik.
\newblock Optimal permutation routing on mesh networks.
\newblock In {\em Proc. of International Network Optimization Conference},
  Belgium, April 2007.
\newblock 6 pages.

\bibitem{opt_hexnet_preprint}
I.~Sau and J.~\v{Z}erovnik.
\newblock An optimal permutation routing algorithm on full-duplex hexagonal
  networks.
\newblock {\em Discrete Mathematics {\&} Theoretical Computer Science},
  10(3):49--62, 2008.

\bibitem{Sch98}
C.~Scheideler.
\newblock {\em {Universal Routing Strategies for Interconnection Networks}}.
\newblock Springer, 1998.

\bibitem{Sybeyn97}
J.~Sibeyn.
\newblock {Routing on Triangles, Tori and Honeycombs}.
\newblock {\em International Journal of Fundations of Computer Science},
  8(3):269--287, 1997.

\bibitem{Sibeyn_1k}
J.~Sibeyn and M.~Kaufman.
\newblock Deterministic 1-k routing on meshes (with applications to worm-hole
  routing).
\newblock {\em Lecture Notes in Computer Science}, 775:237--248, 1994.

\bibitem{sibeyn97deterministic}
J.~F. Sibeyn, B.~S. Chlebus, and M.~Kaufmann.
\newblock {Deterministic Permutation Routing on Meshes}.
\newblock {\em Journal of Algorithms}, 22(1):111--141, 1997.

\bibitem{258658}
A.~Srinivasan and C.-P. Teo.
\newblock A constant-factor approximation algorithm for packet routing, and
  balancing local vs. global criteria.
\newblock In {\em Proceedings of the 20th Annual ACM Symposium on Theory of
  Computing}, pages 636--643, 1997.

\bibitem{honeycomb}
I.~Stojmenovi\'{c}.
\newblock {Honeycomb Networks: Topological Properties and Communication
  Algorithms}.
\newblock {\em IEEE Transactions on Parallel and Distributed Systems},
  8(10):1036--1042, 1997.

\bibitem{Suel_94}
T.~Suel.
\newblock {Routing and Sorting on Meshes with Row and Column Buses}.
\newblock In {\em Proceedings of the 8th International Symposium on Parallel
  Processing}, pages 411--417, 1994.

\bibitem{chemistry}
R.~To\v{s}i\'{c}, D.~Masulovi\'{c}, I.~Stojmenovi\'{c}, J.~Brunvoll, B.~Cyvin,
  and S.~Cyvin.
\newblock {Enumeration of Polyhex Hydrocarbons up to h=17}.
\newblock {\em Journal of Chemical Information and Computer Sciences},
  35:181--187, 1995.

\bibitem{trobec}
R.~Trobec.
\newblock Two-dimensional regular $d$-meshes.
\newblock {\em Parallel Computing}, 26:1945--1953, 2000.

\bibitem{ValBre81}
L.~Valiant and G.~Brebner.
\newblock {Universal schemes for parallel communication}.
\newblock In {\em Proceedings of the 13th Annual ACM Symposium on Theory of
  Computing}, pages 263--277, 1981.

\bibitem{williams1979gfn}
R.~Williams.
\newblock {\em {The Geometrical Foundation of Natural Structure: A Source Book
  of Design}}.
\newblock Dover Publications, 1979.

\end{thebibliography}

\end{document}